\documentclass[10pt]{amsart}

\usepackage{enumerate}
\usepackage{amsmath,amssymb}

\newtheorem{theorem}{Theorem}[section]
\newtheorem{corollary}[theorem]{Corollary}
\newtheorem{lemma}[theorem]{Lemma}
\newtheorem{problem}[theorem]{Problem}
\newtheorem{proposition}[theorem]{Proposition}

\theoremstyle{definition}
\newtheorem{remark}[theorem]{Remark}

\def\B{{\mathcal{B}}}
\def\P{{\mathcal{P}}}
\def\U{{\mathcal{U}}}
\def\Z{{\mathcal{Z}}}

\def\LS{{\mathcal{LS}}}

\usepackage{comment}

\begin{document}

\title[Order isomorphisms]{Order isomorphisms of operator intervals \\
in von Neumann algebras}

\author[M. Mori]{Michiya Mori}

\address{Graduate School of Mathematical Sciences, the University of Tokyo, Komaba, Tokyo, 153-8914, Japan.}
\email{mmori@ms.u-tokyo.ac.jp}



\date{}

\begin{abstract}
We give a complete description of order isomorphisms between operator intervals in general von Neumann algebras. 
For the description, we use Jordan $^*$-isomorphisms and locally measurable operators. 
Our results generalize several works by L. Moln\'ar and P. \v{S}emrl on type I factors. 
\end{abstract}
\maketitle
\thispagestyle{empty}

\section{Introduction}
Let $(X_1, \leq_1)$ and $(X_2, \leq_2)$ be two ordered sets. 
A bijection $\phi\colon X_1\to X_2$ is called an \emph{order isomorphism} if it satisfies $x\leq_1 y\Longleftrightarrow \phi(x)\leq_2\phi(y)$ for all $x, y\in X_1$. 

For a C$^*$-algebra $A$, we use the symbol $A_{sa}$ to mean the collection of all self-adjoint elements in $A$. 
In this paper, we consider order isomorphisms with respect to the natural order structure on $A_{sa}$. 
For an element $a\in A$, we say $a$ is \emph{positive} if there exists an element $x\in A$ such that $a=x^*x$. 
The symbol $A_+$ means the collection of positive elements in $A$. 
It is well-known that when $A$ is realized as a C$^*$-subalgebra of $\B(H)$ (the algebra of bounded linear operators on a complex Hilbert space $H$), $a\in A$ is positive if and only if $a$ satisfies $\langle a\xi, \xi\rangle \geq 0$ for every $\xi\in H$. 
For $a_1, a_2\in A_{sa}$, we write $a_1\leq a_2$ if $a_2-a_1$ is positive. 
Then the relation $\leq$ is an order on $A_{sa}$. 

In 1952, Kadison proved the following. 

\begin{theorem}[Kadison, {\cite[Corollary 5]{K}}]\label{Kadison}
Let $A$ and $B$ be two unital C$^*$-algebras. 
Suppose that $\phi\colon A_{sa}\to B_{sa}$ is an order isomorphism. 
Suppose further that $\phi$ is linear and satisfies $\phi(1)=1$. 
Then, there exists a unique Jordan $^*$-isomorphism $J\colon A\to B$ that extends $\phi$. 
\end{theorem} 

Recall that a linear bijection $J\colon A\to B$ between two C$^*$-algebras is called a \emph{Jordan $^*$-isomorphism} if it satisfies $J(x^*)= J(x)^*$ and $J(x_1x_2+x_2x_1) = J(x_1)J(x_2)+J(x_2)J(x_1)$ for any $x, x_1, x_2\in A$. 
It is easy to verify that a Jordan $^*$-isomorphism between two C$^*$-algebras preserves the order. 

In this paper, we want to weaken the assumption in Kadison's theorem. 
In particular, what if we drop the assumption of linearity?
There are several studies on nonlinear order isomorphisms between self-adjoint parts of commutative C$^*$-algebras (e.g.\ \cite{CS}). 
As we can see in the case $A=B=\mathbb{C}$, such mappings can be far from linear.

However, in noncommutative settings, chances are that one can obtain a conclusion that is completely different from commutative cases. 
In fact, in the special case $A=B=\B(H)$, Moln\'ar gave the results below. 
The collection $E(\B(H)):= \{a\in \B(H)_{sa}\mid 0\leq a\leq 1\}$ is called the \emph{effect algebra} on $H$.

\begin{theorem}[Moln\'ar, {\cite[Theorems 1 and 2]{M01}} and {\cite[Corollary 4]{M03}}]\label{Molnar}
Let $H$ be a complex Hilbert space with $\operatorname{dim}H\geq 2$.
\begin{enumerate}[$(1)$]
\item If $\phi \colon \B(H)_{sa}\to \B(H)_{sa}$ is an order automorphism, then 
there exist a linear or conjugate-linear bounded invertible operator $x$ on $H$ and $b\in \B(H)_{sa}$ that satisfy $\phi(a) = xax^* + b$ for every $a\in \B(H)_{sa}$.
\item If $\phi \colon \B(H)_+\to \B(H)_+$ is an order automorphism, then there exists a linear or conjugate-linear bounded invertible operator $x$ on $H$ that satisfies $\phi(a) = xax^*$ for every $a\in \B(H)_+$.
\item If $\phi\colon E(\B(H))\to E(\B(H))$ is an order automorphism with $\phi(1/2)=1/2$, then there exists a unitary or an antiunitary operator $u$ on $H$ that satisfies $\phi(a) = uau^*$ for every $a\in E(\B(H))$. 
\end{enumerate}
\end{theorem}
 
\v{S}emrl gave a complete description of order isomorphisms between operator intervals of $\B(H)$. 
In particular, he gave a complete description of order automorphisms on the effect algebra. 
Part of what he proved can be summarized in the following: 
\begin{theorem}[\v{S}emrl, \cite{S12}, {\cite[Theorem 2.4]{S17}}]\label{Semrl}
Let $H$ be a complex Hilbert space with $\operatorname{dim}H\geq 2$.
\begin{enumerate}[$(1)$]
\item If $\phi\colon E(\B(H))\to E(\B(H))$ is an order automorphism, then both $\phi(1/2)$ and $1-\phi(1/2)$ are invertible in $\B(H)$. 
\item Conversely, if $b\in E(\B(H))$ and both $b$ and $1-b$ are invertible in $\B(H)$, then there exists an order automorphism $\phi\colon E(\B(H))\to E(\B(H))$ that satisfies $\phi(1/2) = b$. 
\end{enumerate}
\end{theorem}
Actually, \v{S}emrl gave a more concrete description for such mappings. 
For the details, see \cite{S17}, \cite{S18} and \cite{PS}.
Note that an order automorphism on the effect algebra is not necessarily affine.
(In this paper, we say a mapping $\phi\colon \mathcal{C}_1\to \mathcal{C}_2$ between two convex sets is \emph{affine} if it satisfies $\phi(\lambda a+(1-\lambda)b) = \lambda \phi(a)+ (1-\lambda)\phi(b)$ for any $a, b\in \mathcal{C}_1$ and $\lambda\in [0, 1]$.)\medskip

Let $a_1, a_2\in \B(H)_{sa}$ satisfy $a_1\leq a_2$. 
Then, it is easy to see that the subset $\{a\in\B(H)_{sa}\mid a_1\leq a\leq a_2\}$ is totally ordered if and only if $a_2-a_1$ is some scalar multiple of a rank-one projection on $H$. 
In the proofs of the above results by Moln\'ar and \v{S}emrl, rank-one projections on $H$ are crucially important. 
However, it seems to be difficult to use similar ideas in more general settings. 

In this paper, we give a generalization of Moln\'ar's theorem and \v{S}emrl's in the setting of arbitrary von Neumann algebras. 
Although there may be no minimal (or rank-one) projections in such a situation, we can use properties of  projection lattices to obtain our results. 
We heavily depend on the seminal result due to Dye in \cite{D} on orthogonality preserving mappings between projection lattices of two von Neumann algebras. 

Let $M$ be a von Neumann algebra. 
We use the notation $M^{-1}$ to mean the collection of invertible elements in $M$.
Let $a\in M_+$. 
We say that $a$ is \emph{locally measurably invertible} in $M$ if there exists no element $b\in M_+\setminus\{0\}$ with the following property: 
If $x\in M_+$ satisfies $x\leq a$ and $x\leq b$, then $x=0$. 
(In Section \ref{LS}, we show that this is equivalent to the condition that $a$ is invertible in the algebra of locally measurable operators for $M$, 
so that this terminology is consistent.) 
The \emph{effect algebra} of $M$ is the collection $E(M):= \{a\in M_{sa}\mid 0\leq a\leq 1\}$. 
It is easy to see that every order isomorphism between effect algebras of von Neumann algebras preserves the collection of locally measurably invertible elements. 

The following is a summary of main results in this paper (Theorems \ref{Jordan}, \ref{rest} and Proposition \ref{characterization}). 
\begin{theorem}\label{main}
Let $M$ and $N$ be two von Neumann algebras. Assume that $M$ does not admit a type I$_1$ direct summand. 
\begin{enumerate}[$(1)$]
\item If $\phi \colon M_{sa}\to N_{sa}$ is an order isomorphism, then there exist a Jordan $^*$-isomorphism $J\colon M\to N$ and elements $x\in N^{-1}$, $b\in N_{sa}$ that satisfy $\phi(a) = xJ(a)x^* + b$ for every $a\in M_{sa}$. 
\item If $\phi \colon M_+\to N_+$ is an order isomorphism, then there exist a Jordan $^*$-isomorphism $J\colon M\to N$ and an element $x\in N^{-1}$ that satisfy $\phi(a) = xJ(a)x^*$ for every $a\in M_+$. 
\item If $\phi\colon E(M)\to E(N)$ is an order isomorphism with $\phi(1/2)=1/2$, then there exists a Jordan $^*$-isomorphism $J\colon M\to N$ that satisfies $\phi(a) = J(a)$ for every $a\in E(M)$. 
\item Let $\phi\colon E(M)\to E(N)$ be an order isomorphism. Then $M$ is Jordan $^*$-isomorphic to $N$. In addition, both $\phi(1/2)$ and $1-\phi(1/2)$ are locally measurably invertible in $N$. 
\item Conversely, suppose that $M$ and $N$ are Jordan $^*$-isomorphic and $b\in E(N)$. If both $b$ and $1-b$ are locally measurably invertible in $N$, then there exists an order isomorphism $\phi\colon E(M)\to E(N)$ that satisfies $\phi(1/2) = b$. 
\end{enumerate}
\end{theorem}
Note that every Jordan $^*$-isomorphism between two von Neumann algebras decomposes into the direct sum of a $^*$-isomorphism and a $^*$-antiisomorphism. See for example \cite[Exercise 10.5.26]{KR}.

Consider the case $M=N=\B(H)$. 
For an element $b\in \B(H)_+$, $b$ is locally measurably invertible in $\B(H)$ if and only if it is invertible in $\B(H)$. 
In addition, it is known that a linear bijection $\phi$ from $\B(H)$ onto itself is a Jordan $^*$-automorphism if and only if there exists a unitary or an antiunitary operator $u$ on $H$ that satisfies $\phi(a) = uau^*$ for every $a\in\B(H)_{sa}$. 
Thus Theorems \ref{Molnar} and \ref{Semrl} can be obtained as corollaries of Theorem \ref{main}.  

The contents of this paper are as follows. 
Section \ref{pre} is devoted to preliminary results on operator intervals and the algebra of locally measurable operators of a von Neumann algebra. 
Sections \ref{effect} and \ref{interval} are main parts of this paper. 
In Section \ref{effect}, we give a complete description of order isomorphisms between effect algebras of von Neumann algebras. 
We can understand the general form of order isomorphisms between arbitrary operator intervals in von Neumann algebras through Section \ref{interval}. 
Additionally, we give another characterization of order isomorphisms between effect algebras (Proposition \ref{last}). 
Our results in this paper significantly generalize previous studies, but there seems to be much left to be done. 
We present several problems in Section \ref{problem}.

The readers who are interested in this paper are recommended to refer to the book \cite{M} and the articles in the special issue \cite{Acta}, which contain many results and surveys on \emph{preserver problems} with a modern approach.

\section{Preliminaries}\label{pre}
For a general reference on the theory of operator algebras, see for example \cite{T} or \cite{KR}. 
We use the following notation: 
For a von Neumann algebra $M$,  
\begin{itemize}
\item $\P(M):=\{p\in M\mid p=p^2=p^*\}$ means the \emph{projection lattice} of $M$.
\item $\U(M):=\{u\in M\mid uu^*=u^*u=1\}$ means the \emph{unitary group} of $M$. 
\item $\Z(M):=\{x\in M\mid xy=yx \text{ for any }y\in M\}$ is the \emph{center} of $M$. 
\end{itemize}
For $p, q\in \P(M)$, we write 
\begin{itemize}
\item $p\sim q$ and say that $p$ and $q$ are \emph{equivalent} if there exists a partial isometry $v\in M$ such that $vv^*=p$ and $v^*v=q$. 
\item $p\prec q$ if there exists a partial isometry $v\in M$ such that $vv^*=p$ and $v^*v\leq q$. 
\end{itemize}
The symbol $p^{\perp}:= 1-p$ means the orthocomplement projection of $p$. 

\subsection{Operator intervals in von Neumann algebras}\label{pre1}
Let $M$ be a von Neumann algebra. 
For $a_1, a_2\in M_{sa}$, we write $a_1<a_2$ if $a_1\leq a_2$ and $a_2-a_1$ is invertible in $M$. 

For $a_0\in M_{sa}$, we define 
\[
\begin{split}
M_{[a_0, \infty)} &:= \{a\in M_{sa}\mid a_0\leq a\},\\
M_{(a_0, \infty)} &:= \{a\in M_{sa}\mid a_0< a\},\\
M_{(-\infty, a_0]} &:= \{a\in M_{sa}\mid a\leq a_0\} \text{ and}\\
M_{(-\infty, a_0)} &:= \{a\in M_{sa}\mid a< a_0\}.
\end{split}
\]
For $a_1, a_2\in M_{sa}$ with $a_1\leq a_2$ and $a_1\neq a_2$, we define 
\[
M_{[a_1, a_2]} := \{a\in M_{sa}\mid a_1\leq a\leq a_2\}.
\]
In particular, the effect algebra $E(M)$ is equal to $M_{[0, 1]}$.
For $a_1, a_2\in M_{sa}$ with $a_1<a_2$, we define
\[
\begin{split}
M_{[a_1, a_2)} &:= \{a\in M_{sa}\mid a_1\leq a<a_2\},\\
M_{(a_1, a_2]} &:= \{a\in M_{sa}\mid a_1<a\leq a_2\} \text{ and}\\
M_{(a_1, a_2)} &:= \{a\in M_{sa}\mid a_1< a< a_2\}.
\end{split}
\]
A subset $L$ of $M_{sa}$ is called an \emph{operator interval} if it can be written as one of the above forms or $L=M_{sa}$. 

\begin{lemma}\label{normalize}
Let $M$ be a von Neumann algebra. 
\begin{enumerate}[$(1)$]
\item For $a_0\in M_{sa}$, $M_{[a_0, \infty)}$ (resp. $M_{(-\infty, a_0]}$) is order isomorphic to $M_{[0, \infty)}$ (resp. $M_{(-\infty, 0]}$). 
\item For $a_0\in M_{sa}$, both $M_{(a_0, \infty)}$ and $M_{(-\infty, a_0)}$ are order isomorphic to $M_{(0, \infty)}$. 
\item Suppose that $a_1, a_2\in M_{sa}$ satisfy $a_1\leq a_2$ and $a_1\neq a_2$. 
Take the support projection $p\in \P(M)$ of $a_2-a_1$. 
Then $M_{[a_1, a_2]}$ is order isomorphic to $E(pMp)$. 
\item For $a_1, a_2\in M_{sa}$ with $a_1<a_2$, 
$M_{[a_1, a_2)}$ (resp. $M_{(a_1, a_2]}$, $M_{(a_1, a_2)}$) is order isomorphic to  
$M_{[0, \infty)}$ (resp. $M_{(-\infty, 0]}$, $M_{(0, \infty)}$). 
\end{enumerate}
\end{lemma}
\begin{proof}
$(2)$ By a shift, $M_{(a_0, \infty)}$ (resp. $M_{(-\infty, a_0)}$) is order isomorphic to $M_{(0, \infty)}$ (resp. $M_{(-\infty, 0)}$). 
By the mapping $a\mapsto (1-a)^{-1}-1$ (resp. $a\mapsto -a^{-1}+1$), $M_{(0, 1)}$ is order isomorphic to $M_{(0, \infty)}$ (resp. $M_{(-\infty, 0)}$). 

$(3)$ It is easy to see that the mapping $E(pMp)\ni a\mapsto (a_2-a_1)^{1/2}a(a_2-a_1)^{1/2} +a_1$ gives an order isomorphism from $E(pMp)$ onto $M_{[a_1, a_2]}$.

We can prove $(1)$ and $(4)$ similarly and so we omit the proof of them. 
\end{proof}

By this lemma, every operator interval of a von Neumann algebra is order isomorphic to $M_{sa}$, $E(M)$, $M_{(0, \infty)}$, $M_{(-\infty, 0]}$ or $M_{[0, \infty)}$ for some von Neumann algebra $M$. 
Thus, in order to understand the general form of order isomorphisms between two operator intervals, it suffices to consider only operator intervals of these forms. 
Note that $E(M)$ and $M_{(-\infty, 0]}$ (resp. $E(M)$ and  $M_{[0, \infty)}$) have maximum (resp. minimum) elements but the others do not.

\subsection{Locally measurable operators and their invertibility}\label{LS}

Let $M\subset \B(H)$ be a von Neumann algebra. 
A densely-defined closed operator $X\colon \operatorname{dom}(X)\,(\subset H)\,\to H$ is said to be \emph{affiliated with $M$} if $yX\subset Xy$ for every element $y$ in $M'$, the commutant of $M$. 

Let $X$ be an operator affiliated with $M$. 
Put $T:= (X^*X)^{1/2}$. 
We say $X$ is \emph{measurable} if the spectral projection $\chi_{(n, \infty)} (T)\in \P(M)$ is a finite projection in $M$ for some positive integer $n$. 
We say $X$ is \emph{locally measurable} if there exists an increasing sequence $\{p_n\}_{n\geq 1}\subset \P(\Z(M))$ of central projections in $M$ such that $\bigvee_{n\geq 1} p_n = 1$ and $Xp_n$ is measurable in $M$ for any $n\geq 1$. 
We use the notation $\mathcal{S}(M)$ (resp. $\LS(M)$) for the collection of measurable (resp. locally measurable) operators affiliated with $M$. 

The following are known: 
\begin{itemize}
\item If $X, Y\in\mathcal{S}(M)$ (resp. $X, Y\in\LS(M)$), then $X^*$ and the closures of $XY$ and $X+Y$ are in $\mathcal{S}(M)$ (resp. $\LS(M)$). 
\item In that way, $\mathcal{S}(M)$ and $\LS(M)$ can be considered as $^*$-algebras which contain $M$ as a $^*$-subalgebra. 
\end{itemize}
See \cite{Y} for the details. 

In order to prove the main theorem of this paper, we make use of the lemma below. 
The author suspects that this result is already known, but we give a concise proof for completeness. 

\begin{lemma}\label{invertible}
Let $M$ be a von Neumann algebra and $a\in M_{+}$. 
Take the central projections $p_i\in \P(\Z(M))$, $i = I, II, III$, which are determined by the condition that $p_I+p_{II}+p_{III} = 1$ and either $p_i=0$ or $Mp_i$ is of type $i$, $i=I, II, III$. 
Then the following three conditions are equivalent. 
\begin{enumerate}[$(1)$]
\item The operator $a$ is invertible in the algebra $\LS(M)$. 
\item There exists no element $b\in M_+\setminus\{0\}$ with the following property: 
If $x\in M_+$ satisfies $x\leq a$ and $x\leq b$, then $x=0$. 
\item There exists a sequence $\{q_n\}_{n\geq 1}\subset \P(\Z(M))$ of central projections in $M$ such that $\sum_{n\geq 1} q_n = 1$, either $q_n (p_I+p_{III})=0$ or $aq_n (p_I+p_{III})$ is invertible in $Mq_n (p_I+p_{III})$, and either $q_n p_{II}=0$ or $aq_n p_{II}$ is an invertible element in $\mathcal{S}(Mq_n p_{II})$, $n\geq 1$. 
\end{enumerate} 
\end{lemma}

Thus we say $a\in M_+$ is \emph{locally measurably invertible} in $M$ if one of the above equivalent conditions is satisfied. 
For the proof of this lemma, we need an additional lemma. 

\begin{lemma}\label{lem0}
Let $M$ be a von Neumann algebra and $p\in M$ be a finite projection. 
Suppose that an increasing sequence $\{p_n\}_{n\geq1}$ of projections in $M$ satisfies $\bigvee_{n\geq1} p_n \succ p$. 
Then there exists an increasing sequence $\{\tilde{p}_n\}_{n\geq1}$ of projections in $M$ such that $\tilde{p}_n\leq p_n$ and $\bigvee_{n\geq1}\tilde{p}_n \sim p$.
\end{lemma}
\begin{proof}
For $n\geq 1$, take the maximum central projection $e_n\in \P(\Z(M))$ that  satisfies $p_ne_n \succ pe_n $. 
Then $\{e_n\}_{n\geq 1}$ is an increasing sequence and $p_ne_n^{\perp} \prec pe_n^{\perp}$. 
Put $e:= \bigvee_{n\geq 1}e_n$.
Take a sequence $\{q_n\}_{n\geq1}\subset \P(M)$ such that $pe_1\sim q_1\leq p_1e_1$ and $p(e_n-e_{n-1})\sim q_n\leq p_n(e_n-e_{n-1})$ for $n\geq 2$. 
Put $\tilde{p}_n:= \sum_{k=1}^n q_k + p_ne^{\perp}\,\,(\leq p_n)$. 
Then $\{\tilde{p}_n\}_{n\geq1}$ is an increasing sequence and satisfies $\bigvee_{n\geq 1}\tilde{p}_n e \sim pe$. 
The sequence $\{\tilde{p}_n e^{\perp}\}_{n\geq1}$ is an increasing sequence and satisfies $\tilde{p}_n e^{\perp}\prec pe^{\perp}$, $n\geq 1$. 
Take a projection $\hat{p}_1\in \P(M)$ such that $\tilde{p}_1e^{\perp}\sim \hat{p}_1\leq pe^{\perp}$. 
By finiteness of $pe^{\perp}$, we can take a sequence $\{\hat{p}_n\}_{n\geq2}$ of projections in $M$ such that 
$\{\hat{p}_n\}_{n\geq 1}$ is mutually orthogonal and $(\tilde{p}_n - \tilde{p}_{n-1})e^{\perp} \sim \hat{p}_n\leq pe^{\perp}$, $n\geq 2$. 
Then $\bigvee_{n\geq 1} \tilde{p}_n e^{\perp} \sim \sum_{n\geq 1} \hat{p}_n \leq pe^{\perp}$. 
Since $\bigvee_{n\geq 1}\tilde{p}_n e^{\perp} = \bigvee_{n\geq 1}p_n e^{\perp}\succ pe^{\perp}$, we have $\bigvee_{n\geq 1}\tilde{p}_n e^{\perp}\sim pe^{\perp}$, and thus $\bigvee_{n\geq1}\tilde{p}_n \sim p$.
\end{proof}

\begin{proof}[Proof of Lemma \ref{invertible}]
$(3)\Rightarrow(1)$ We can take $b_n \in \mathcal{S}(Mq_n)$ such that $b_naq_n=q_n$ for each $n\geq 1$. 
Then the sum $\sum_{n\geq 1} b_n$ is the inverse of $a$ in $\LS(M)$. 

$(1)\Rightarrow(2)$ 
Take the inverse $a^{-1}\in \LS(M)$ of $a$ and its positive square root $a^{-1/2}\in\LS(M)$. 
The mapping $x\mapsto a^{-1/2}xa^{-1/2}$ is an order isomorphism from $M_+$ onto $\{x\in \LS(M)\mid 0\leq x\leq a^{-1}\}$ and $a$ is mapped to $1$. 
We define a function $f\colon [0, \infty)\to \mathbb{R}$ by $f(t):= \min\{t, 1\}$, $t\in [0, \infty)$. 
For every $0\neq b \in \LS(M)$ with $0\leq b\leq a^{-1}$, the element $(0\neq)\,\, x:=f(b) \in M_+$ satisfies both $x\leq b$ and $x\leq 1$. 
Thus the condition $(2)$ holds. 

$(2)\Rightarrow(3)$ 
If we decompose $M$ into a direct sum, it suffices to consider each direct summand. 
First we consider the cases of type I or III. 
It suffices to show that $\bigwedge_{n\geq1} z(\chi_{[0, 1/n)}(a)) = 0$, where $z(p)\in \P(\Z(M))$ means the central support of $p$ for a projection $p\in \P(M)$. 
Assume $r:= \bigwedge_{n\geq1} z(\chi_{[0, 1/n)}(a))\neq 0$. 
Considering the pair $(Mr, ar)$ instead of $(M, a)$, we may assume $\bigwedge_{n\geq1} z(\chi_{[0, 1/n)}(a))=1$. 
Take a normal (tracial) state $\tau$ on $\Z(M)$. 
We may also assume $\operatorname{supp}(\tau)=1 \in \P(\Z(M))$. 
By our assumptions, we may take a strictly decreasing sequence $\{c_n\}_{n\geq1}$ of positive real numbers that satisfies $c_1>\lVert a \rVert$, $c_n\to 0$ ($n\to0$) and $\tau(z(\chi_{[c_{n+1}, c_n)}(a))) \geq 1-3^{-n}$, $n\geq 1$. 
Then $\tau(\bigwedge_{n\geq1} z(\chi_{[c_{n+1}, c_n)}(a))) \geq 1-\sum_{n\geq 1}3^{-n}>0$ and thus $\bigwedge_{n\geq1} z(\chi_{[c_{n+1}, c_n)}(a)) \neq 0$. 
We may assume $\bigwedge_{n\geq1} z(\chi_{[c_{n+1}, c_n)}(a))=1$. 

If $M$ is of type I, we can take an abelian projection $p_n\leq \chi_{[c_{n+1}, c_n)}(a)$ with $z(p_n)=1$ for each $n\geq 1$. 
If $M$ is of type III, by the assumption that $\Z(M)$ has a normal faithful state,  we can take a countably decomposable projection $p_n\leq \chi_{[c_{n+1}, c_n)}(a)$ with $z(p_n)=1$ for each $n\geq 1$. 
In both cases, $\{p_n\}_{n\geq1}$ is a family of mutually orthogonal equivalent nonzero projections. 
We consider the operator $\tilde{a} := \sum_{n\geq 1} c_n p_n + c_1 (1-\sum_{n\geq 1} p_n)$. 
We have $a \leq \sum_{n\geq 1} c_n \chi_{[c_{n+1}, c_n)}(a) \leq \tilde{a}$, so $\tilde{a}$ also satisfies the condition $(2)$. 
We can identify $\sum_{n\geq 1} c_n p_n \in (\sum_{n\geq 1} p_n) M  (\sum_{n\geq 1} p_n)$ with $(\sum_{n\geq 1} c_n e_n) \otimes p_1 \in \B(\ell_2)\mathbin{\overline{\otimes}}p_1 M p_1$, where $e_n$ is the projection onto $n$-th coordinate of $\ell_2 =\ell_2(\mathbb{N})$, $\mathbb{N} = \{1, 2,\ldots\}$. 
Put $T:= \sum_{n\geq 1} c_n e_n\in \B(\ell_2)$. 
Since $c_n\to0$ as $n\to\infty$, the positive operator $T\in \B(\ell_2)$ is not invertible. 
Thus there exists a vector $\xi\in \ell_2\setminus T\ell_2$. 
Take the projection $e\in \B(\ell_2)$ whose range is $\mathbb{C}\xi$. 
It is easy to see that, if $x\in \B(\ell_2)_+$ satisfies $x\leq T$ and $x\leq e$, then $x=0$. 
Take the projection $p\in (\sum_{n\geq 1} p_n) M  (\sum_{n\geq 1} p_n)$ that corresponds to $e\otimes p_1\in \B(\ell_2)\mathbin{\overline{\otimes}}p_1 M p_1$. 
By \cite[Proposition 11.2.24]{KR}, 
there exists a family $\{\Phi_i\colon \B(\ell_2)\mathbin{\overline{\otimes}}p_1 Mp_1 \to \B(\ell_2) \otimes \mathbb{C}p_1\}_{i\in I}$ of normal conditional expectations with the property that 
if $x\in (\B(\ell_2)\mathbin{\overline{\otimes}}p_1 Mp_1)_+$ satisfies $\Phi_i(x)=0$ for every $i\in I$, then $x=0$. 
Suppose that $x\in M_+$ satisfies both $x\leq \tilde{a}$ and $x\leq p$. 
It follows that $x\in (\sum_{n\geq 1} p_n) M  (\sum_{n\geq 1} p_n)$ and 
$\Phi_i(x) \leq T\otimes p_1$, $\Phi_i(x) \leq e\otimes p_1$ in $ \B(\ell_2) \otimes \mathbb{C}p_1$ for every $i\in I$. 
Thus we have $\Phi_i(x) = 0$ for every $i\in I$ and hence $x=0$. 
We obtain a contradiction. 

Next we consider the case where $M$ is of type II. 
For a projection $p\in \P(M)$, we take the central projection $z_{\mathrm{infin}}(p)\in\P(\Z(M))$ that is defined as the maximum projection in $\{z\in \P(\Z(M))\mid pz \text{ is properly infinite},\,\, z\leq z(p)\}$. 
It suffices to show that $\bigwedge_{n\geq1} z_{\mathrm{infin}}(\chi_{[0, 1/n)}(a)) = 0$. 
Hence we assume that $\bigwedge_{n\geq1} z_{\mathrm{infin}}(\chi_{[0, 1/n)}(a))=1$. 
Take a normal semifinite faithful tracial weight $\tau$ on $M$ with $\tau(1)\geq 1$ and a (finite) projection $p\in M$ with $\tau(p)=1$. 
It follows that $\chi_{[0, 1/n)}(a) \succ p$ for every $n\geq 1$.
By Lemma \ref{lem0}, there exist a strictly decreasing sequence $\{c_n\}_{n\geq 1}$ of positive real numbers and a sequence $\{p_n\}_{n\geq 1}$ of projections in $M$ such that $c_1> \lVert a\rVert$, $c_n\to 0\,\,(n\to \infty)$, $p_n\leq \chi_{[c_{n+1}, c_n)}(a)$, $p_n\prec p$ and $\tau(p_n)\geq 1-3^{-n}$, $n\geq 1$. 
Take a projection $\tilde{p}_n\in \P(M)$ such that $p_n\sim \tilde{p}_n \leq p$, $n\geq 1$. 
Put $\tilde{p} := \bigwedge_{n\geq 1} \tilde{p}_n$. 
Then $\tau(\tilde{p}) = \tau(\bigwedge_{n\geq 1} \tilde{p}_n) \geq 1-\sum_{n\geq 1} 3^{-n}>0$. 
Hence $\tilde{p} \neq 0$.
Take a projection $\hat{p}_n\in \P(M)$ such that $\hat{p}_n\leq p_n$ and $\hat{p}_n\sim\tilde{p}$, for $n\geq 1$. 
Then $\{\hat{p}_n\}_{n\geq1}$ is a family of mutually orthogonal equivalent nonzero projections. 
Put $\tilde{a} := \sum_{n\geq 1} c_n \hat{p}_n + c_1 (1-\sum_{n\geq 1} \hat{p}_n)\,\, (\geq a)$. 
By a discussion similar to that in the preceding paragraph, we can obtain a contradiction. 
\end{proof}

\section{Order isomorphisms between effect algebras}\label{effect}
In this section, we give a complete description of order isomorphisms between two effect algebras in von Neumann algebras. 
We generalize some of the arguments on $\B(H)$ by \v{S}emrl in \cite{S17}. 

Let $M$ be a von Neumann algebra. 
For $\alpha\in \mathbb{R}$ with $\alpha<1$, we define a continuous monotone increasing bijective function $f_{\alpha}\colon [0, 1]\to [0, 1]$ by 
\[
f_{\alpha}(t):= \frac{t}{\alpha t+1-\alpha},\,\, t\in [0, 1]. 
\]
Then the mapping $a\mapsto f_{\alpha}(a)$, $a\in E(M)$, is an order automorphism on $E(M)$. 
Indeed, $f_0$ is the identity mapping. 
If $\alpha \neq 0$, then we have 
\[
f_\alpha(a) = \frac{1}{\alpha} - \frac{1-\alpha}{\alpha^2} \left(a + \frac{1-\alpha}{\alpha}\right)^{-1} 
\]
for any $a\in E(M)$. 
Thus, for $a_1, a_2\in E(M)$, we have
\[
\begin{split}
a_1\leq a_2 &\Longleftrightarrow \left(a_1 + \frac{1-\alpha}{\alpha}\right)^{-1} \geq \left(a_2 + \frac{1-\alpha}{\alpha}\right)^{-1}\\ 
&\Longleftrightarrow f_{\alpha}(a_1)\leq f_{\alpha}(a_2). 
\end{split}
\]
We also remark that, if $0<\alpha<1$, then we can extend the domain of $f_{\alpha}$ to $\mathbb{R}_+$ by the same formula, and the mapping $a\mapsto f_{\alpha}(a)$, $a\in M_+$ preserves the order. We use this fact in Proposition \ref{pq}.
 
It is easy to see the following:
\begin{itemize}
\item If $\alpha\neq 0$, then for $a\in E(M)$, we have $a\in\P(M)$ if and only if $f_{\alpha}(a)=a$. 
\item If $0\leq \alpha<1$, then we have $f_{\alpha}(a)\geq a$ for any $a\in E(M)$. 
\item If $\alpha\leq0$, then we have $f_{\alpha}(a)\leq a$ for any $a\in E(M)$. 
\end{itemize} 

The lemma below gives us a light on a strategy to attack on general von Neumann algebras. 

\begin{lemma}\label{lem-proj}
Let $M$ be a von Neumann algebra. 
For a pair $(a, b)\in E(M)\times E(M)$ we consider the following condition: 
\begin{itemize}
\item[$(*)$] If $x\in E(M)$ satisfies $x\leq a$ and $x\leq b$, then $x=0$. 
Moreover, if $x\in E(M)$ satisfies $x\geq a$ and $x\geq b$, then $x=1$.
\end{itemize}
For $a\in E(M)$, the following conditions are equivalent: 
\begin{enumerate}[$(1)$]
\item $a\in\P(M)$.  
\item There exists an element $b\in E(M)$ with the following properties: 
\begin{itemize}
\item The pair $(a, b)$ satisfies the condition $(*)$. 
\item Suppose $(a_0, b_0)\in E(M)\times E(M)$ satisfies $(*)$. 
If  $a\leq a_0$ and $b\leq b_0$, then $(a, b)=(a_0, b_0)$. 
\end{itemize}
\end{enumerate}
\end{lemma}
\begin{proof}
$(1)\Rightarrow (2)$ Let $a=p\in\P(M)$. 
Put $b:=p^{\perp}=1-p$. 
If $x\in E(M)$ satisfies $x\leq p$, then $0\leq p^{\perp}xp^{\perp}\leq p^{\perp}pp^{\perp} =0$ and thus $p^{\perp}xp^{\perp}=0$. 
If in addition $x\leq p^{\perp}$, then we also have $pxp = 0$. 
By positivity of $x$, we have $x=0$. 
Similarly, if $x\in E(M)$ satisfies both $x\geq p$ and $x\geq p^{\perp}$, then $x=1$. 
Thus the pair $(a, b)=(p, p^{\perp})$ satisfies the condition $(*)$. 

Suppose that $(a_0, b_0)\in E(M)\times E(M)$ satisfies $a\leq a_0$, $b\leq b_0$ and $a\neq a_0$. 
Since $a=p$ is a projection and $a_0\leq 1$, we have $a_0 = p + p^{\perp}a_0p^{\perp}$ and $p^{\perp}a_0p^{\perp}\neq 0$. 
We have $p^{\perp}a_0p^{\perp}\leq a_0$ and $p^{\perp}a_0p^{\perp}\leq p^{\perp}=b\leq b_0$. 
Hence the pair $(a_0, b_0)$ does not satisfy the condition $(*)$. 
Similar arguments complete the proof of $(1)\Rightarrow (2)$. 

$(2)\Rightarrow (1)$ Suppose that the pair $(a, b)$ satisfies the condition in $(2)$. 
Fix a positive real number $0<\alpha<1$. 
Since $f_{\alpha}\colon E(M)\to E(M)$ is an order isomorphism, we easily see that the pair $(f_{\alpha}(a), f_{\alpha}(b))$ also satisfies $(*)$. 
We also have $a\leq f_{\alpha}(a)$ and $b\leq f_{\alpha}(b)$. 
It follows $(a, b)=(f_{\alpha}(a), f_{\alpha}(b))$. 
Since $\alpha\neq 0$, we obtain $a, b\in \P(M)$.
\end{proof}

\begin{corollary}
Let $M$ and $N$ be von Neumann algebras. 
Suppose that $\phi\colon E(M)\to E(N)$ is an order isomorphism. 
Then $\phi(\P(M))=\P(N)$. 
\end{corollary}
\begin{proof}
By the preceding lemma, for $a\in E(M)$, we have $a\in \P(M)$ if and only if the condition $(2)$ holds. 
It is easy to see that the condition $(2)$ is preserved by an order isomorphism.
\end{proof}

We will use the following four lemmas. 
\begin{lemma}\label{lem1}
Let $M$ be a von Neumann algebra. 
Suppose that $p\in \P(M)$ and $a\in E(M)$ satisfy $pa=0$. 
Then $\{b\in E(M)\mid b\geq p,\,\,b\geq a\}$ admits $p+a$ as its minimum element. 
\end{lemma}
\begin{proof}
Let $x\in E(M)$ satisfy $x\geq p$ and $x\geq a$. 
It follows $p\leq pxp \leq p1p = p$, hence $pxp=p$. 
Since $\lVert x\rVert \leq 1$, we have $px=xp=p$. 
In addition, we have $p^{\perp}xp^{\perp}\geq p^{\perp}ap^{\perp} = a$. 
Thus $x= pxp+p^{\perp}xp^{\perp}\geq p+a$. 
\end{proof}

\begin{lemma}\label{lem2}
Let $M$ and $N$ be two von Neumann algebras and $\phi\colon E(M)\to E(N)$ be an order isomorphism. 
Suppose that $p\in \P(\Z(M))$ and $q\in \P(\Z(N))$ are central projections satisfying $\phi(p)=q$ and $\phi(p^{\perp})=q^{\perp}$. 
Then there exist two order isomorphisms $\phi_1\colon E(Mp)\to E(Nq)$ and $\phi_2\colon E(Mp^{\perp})\to E(Nq^{\perp})$ that satisfy 
\[
\phi(a) = \phi_1(ap) + \phi_2(ap^{\perp}),\quad a\in E(M). 
\]
\end{lemma}
\begin{proof}
By the three conditions $\phi(p)=q$, $\phi(p^{\perp})=q^{\perp}$ and $\phi(0)=0$, we can construct order isomorphisms $\phi_1\colon E(Mp)\to E(Nq)$ and $\phi_2\colon E(Mp^{\perp})\to E(Nq^{\perp})$ by $\phi_1(x)=\phi(x)$, $x\in E(Mp)\,(\subset E(M))$ and $\phi_2(x)=\phi(x)$, $x\in E(Mp^{\perp})\,(\subset E(M))$. 
Let $a\in E(M)$. 
Then the assumption that $p$ is a central projection assures that $a$ is the minimum element in $\{x\in E(M)\mid x\geq ap,\,\, x\geq ap^{\perp}\}$. 
Since $\phi$ is an order isomorphism, $\phi(a)$ is the minimum element in $\{x\in E(N)\mid x\geq \phi(ap)\, (=\phi_1(ap)),\,\,x\geq \phi(ap^{\perp})\, (=\phi_2(ap^{\perp}))\}$. 
By the assumption that $q$ is central, we have $\phi(a) = \phi_1(ap) + \phi_2(ap^{\perp})$. 
\end{proof}

\begin{lemma}\label{lem3}
Let $M\subset \B(H)$ be a von Neumann algebra. 
Suppose that $p, q\in \P(M)$ satisfy $p\wedge q= 0$, $p\vee q=1$ and the following property: 
If $x\in E(M)$ satisfies $x\geq p$ and $x\geq q/2$, then $x\geq 1/2$. 
Then we have $pq = 0$. 
\end{lemma}
\begin{proof}
By Halmos's two projection theorem \cite{H}, we may assume the following: 
We have a decomposition $H= H_1\oplus H_2\oplus K_1\oplus K_2$ of the Hilbert space $H$ with an identification $K_1=K_2$ and positive injective operators $a, b\in \B(K_1)$ with $a^2+b^2=1\in \B(K_1)$ such that $p$ and $q$ are identified with:
\[
p=\begin{pmatrix}1&0&0&0\\0&0&0&0\\0&0&1&0\\0&0&0&0\end{pmatrix},\quad
q=\begin{pmatrix}0&0&0&0\\0&1&0&0\\0&0&a^2&ab\\0&0&ab&b^2\end{pmatrix}. 
\]
Note that $pq=0 \Longleftrightarrow K_1=0=K_2$. 
Consider the operator $x_0:=p + f(p^{\perp}qp^{\perp}) \in E(M)$, where $f$ is a mapping determined by $f(t):= t/(1+t)$, $t\in [0, 1]$. Then $x_0\geq p$. 
We have 
\[
x_0=\begin{pmatrix}1&0&0&0\\0&1/2&0&0\\0&0&1&0\\0&0&0&f(b^2)\end{pmatrix} 
\]
and hence
\[
\begin{split}
2x_0-q &= \begin{pmatrix}2&0&0&0\\0&0&0&0\\0&0&2-a^2&-ab\\0&0&-ab&2f(b^2)-b^2\end{pmatrix}\\
&=\begin{pmatrix}2&0&0&0\\0&0&0&0\\0&0&0&0\\0&0&0&0\end{pmatrix}+
\begin{pmatrix}0&0&0&0\\0&0&0&0\\0&0&b^2+1&-b(1-b^2)^{1/2}\\0&0&-b(1-b^2)^{1/2}&2f(b^2)-b^2\end{pmatrix}.
\end{split}
\] 
In addition, we have $2f(b^2)-b^2\geq 0$ and
\[
\begin{split}
\B(K_1\oplus K_2)&\ni 
\begin{pmatrix}b^2+1&-b(1-b^2)^{1/2}\\-b(1-b^2)^{1/2}&2f(b^2)-b^2\end{pmatrix}\\
&=\begin{pmatrix}(b^2+1)^{1/2}&0\\-(2f(b^2)-b^2)^{1/2}&0\end{pmatrix}
{\begin{pmatrix}(b^2+1)^{1/2}&0\\-(2f(b^2)-b^2)^{1/2}&0\end{pmatrix}}^{*}\geq 0, 
\end{split}
\]
thus $x_0\geq q/2$. 
By the assumption, we obtain $x_0\geq 1/2$, which means $f(b^2)\geq 1/2$ and thus $K_1=0=K_2$. 
\end{proof}

\begin{lemma}\label{lem4}
Let $M$ be a von Neumann algebra. 
Suppose that a pair $(a, b)\in E(M)\times E(M)$ satisfies $a^2+b^2 = 1$. 
(Then $a$ and $b$ commute.) 
Suppose further that $\lambda\in \mathbb{R}$ satisfies $\lambda>1$ and $x\in M_+$, $u\in \U(\Z(M))$. 
We consider $M_2(M)$ (two by two matrices with entries in $M$). 
Then the condition 
\[
\lambda\begin{pmatrix}x&0\\0&1\end{pmatrix} \geq 
\begin{pmatrix}a^2&abu\\abu^*&b^2\end{pmatrix}\quad \text{in }M_2(M)
\]
is equivalent to $\lambda x \geq a^2b^2(\lambda-b^2)^{-1}+a^2$ in $M$. 
\end{lemma}
\begin{proof}
We have
\[
\begin{split}
&\lambda\begin{pmatrix}x&0\\0&1\end{pmatrix} \geq \begin{pmatrix}a^2&abu\\abu^*&b^2\end{pmatrix}\\
&\Longleftrightarrow\lambda\begin{pmatrix}1&0\\0&u\end{pmatrix}\begin{pmatrix}x&0\\0&1\end{pmatrix} \begin{pmatrix}1&0\\0&u^*\end{pmatrix}\geq 
\begin{pmatrix}1&0\\0&u\end{pmatrix}\begin{pmatrix}a^2&abu\\abu^*&b^2\end{pmatrix}\begin{pmatrix}1&0\\0&u^*\end{pmatrix}\\
&\Longleftrightarrow\lambda\begin{pmatrix}x&0\\0&1\end{pmatrix} \geq 
\begin{pmatrix}a^2&ab\\ab&b^2\end{pmatrix}\\
&\Longleftrightarrow 
\begin{pmatrix}\lambda x-a^2&-ab\\-ab&\lambda-b^2\end{pmatrix}\geq 0
\end{split}
\]
and
\[
\begin{split}
&\begin{pmatrix}\lambda x-a^2&-ab\\-ab&\lambda-b^2\end{pmatrix}\geq 0\\
&\Longrightarrow
\begin{pmatrix}(\lambda-b^2)^{1/2}&ab(\lambda-b^2)^{-1/2}\end{pmatrix}\begin{pmatrix}\lambda x-a^2&-ab\\-ab&\lambda-b^2\end{pmatrix}
\begin{pmatrix}(\lambda-b^2)^{1/2}\\(\lambda-b^2)^{-1/2}ab\end{pmatrix}\geq 0\\
&\Longleftrightarrow 
(\lambda-b^2)^{1/2}(\lambda x-a^2)(\lambda-b^2)^{1/2} \geq a^2b^2\\
&\Longleftrightarrow 
\lambda x-a^2 \geq (\lambda-b^2)^{-1/2} a^2b^2 (\lambda-b^2)^{-1/2}\\
&\Longleftrightarrow 
\lambda x \geq a^2b^2(\lambda-b^2)^{-1}+a^2.
\end{split}
\]
Conversely, if $\lambda x \geq a^2b^2(\lambda-b^2)^{-1}+a^2$, then 
\[
\begin{split}
\begin{pmatrix}\lambda x-a^2&-ab\\-ab&\lambda-b^2\end{pmatrix}
&\geq\begin{pmatrix}a^2b^2(\lambda-b^2)^{-1}&-ab\\-ab&\lambda-b^2\end{pmatrix}\\
&= \begin{pmatrix}ab(\lambda-b^2)^{-1/2}\\-(\lambda-b^2)^{1/2}\end{pmatrix}
\begin{pmatrix}ab(\lambda-b^2)^{-1/2}&-(\lambda-b^2)^{1/2}\end{pmatrix}\geq 0.
\end{split}
\]
\end{proof}

We recall Dye's theorem on orthoisomorphisms between projection lattices.
Let $M$ and $N$ be two von  Neumann algebras. 
A bijection $\phi$ from $\P(M)$ onto $\P(N)$ is called an \emph{orthoisomorphism} when it satisfies $pq = 0$ if and only if $\phi(p)\phi(q) = 0$, for $p, q\in \P(M)$.
\begin{theorem}[Dye, {\cite[Corollary of Theorem 1]{D}}]
Let $M$ and $N$ be two von Neumann algebras. 
Suppose that $M$ does not admit a type I$_2$ direct summand and that $\phi\colon \P(M)\to \P(N)$ is an orthoisomorphism.
Then there exists a Jordan $^*$-isomorphism from $M$ onto $N$ that extends $\phi$. 
\end{theorem}

Now we are in the position of proving one of the main results in this paper. 
This gives a generalization of \cite[Corollary 4]{M03}, in which Moln\'ar considered the case of type I factors. 
The author obtained a hint of this proof partly from \cite[Section 4]{S17} by \v{S}emrl. 
\begin{theorem}\label{Jordan}
Let $M$ and $N$ be two von Neumann algebras. 
Suppose that $\phi\colon E(M)\to E(N)$ is an order isomorphism that satisfies $\phi(1/2)=1/2$.
Then $\phi|_{\P(M)}\colon \P(M)\to \P(N)$ is an orthoisomorphism. 
Moreover, if in addition $M$ does not admit a type I$_1$ direct summand, then $\phi$ extends uniquely to a Jordan $^*$-isomorphism from $M$ onto $N$. 
\end{theorem}
\begin{proof}
Suppose $p\in \P(M)$. 
Then the pair $(\phi(p), \phi(p^{\perp}))$ satisfies the condition in $(2)$ of Lemma \ref{lem-proj}. 
It follows that $\phi(p), \phi(p^{\perp})\in \P(N)$ and $\phi(p)\wedge \phi(p^{\perp})= 0$, $\phi(p)\vee \phi(p^{\perp}) = 1$. 
Since $p^{\perp}/2$ is the maximum element in $\{y\in E(M)\mid y\leq p^{\perp},\,\, y\leq 1/2\}$, $\phi(p^{\perp}/2)$ is the maximum element in $\{y\in E(N)\mid y\leq \phi(p^{\perp}),\,\, y\leq 1/2\}$. 
Thus we have $\phi(p^{\perp}/2) = \phi(p^{\perp})/2$. 
By Lemma \ref{lem1}, if $a\in E(M)$ satisfy $a\geq p$ and $a\geq p^{\perp}/2$, then $a\geq 1/2$. 
It follows that if $b\in E(N)$ satisfies $b\geq \phi(p)$ and $b\geq \phi(p^{\perp}/2)\,(= \phi(p^{\perp})/2)$, then $b\geq 1/2$. 
By Lemma \ref{lem3}, we have $\phi(p^{\perp})=\phi(p)^{\perp}$, which completes the first half of this proof. 
In particular, $\phi$ maps a central projection in $M$ to some central projection in $N$ (\cite[Lemma 1]{D}). 
Thus we may apply Lemma \ref{lem2} to decompose $\phi$ into a direct sum. 

Suppose that $M$ does not admit a type I$_1$ direct summand. 
Then $N$ does not, either.  
Take the central projection $r\in \P(\Z(M))$ such that $Mr$ is of type I$_2$ and $Mr^{\perp}$ does not have a type I$_2$ direct summand. 
It follows that $N\phi(r)$ is of type I$_2$ and $N\phi(r)^{\perp}$ does not have a type I$_2$ direct summand.

We first consider I$_2$ cases. 
Suppose that $M$ and $N$ are of type I$_2$. 
Since there exists an orthoisomorphism between projection lattices, 
it follows that $M$ and $N$ are $^*$-isomorphic. 
Let $p\in M$ be a maximal abelian projection in $M$. 
Then so is $\phi(p)$ in $N$. 
We would like to show that $\phi(cp) =c\phi(p)$ for an arbitrary real number $0\leq c\leq 1$. 

We may assume that $M$ and $N$ are identified with $M_2(A)$ for some abelian von Neumann algebra $A$, and that 
\[
p=\begin{pmatrix}1&0\\0&0\end{pmatrix}\in M_2(A)=M,\quad 
\phi(p)=\begin{pmatrix}1&0\\0&0\end{pmatrix}\in M_2(A)=N.
\]
Since $\displaystyle\frac{1}{2}\begin{pmatrix}1&1\\1&1\end{pmatrix}$ is a maximal abelian projection in $M$, so is $\displaystyle\phi\left(\frac{1}{2}\begin{pmatrix}1&1\\1&1\end{pmatrix}\right)$ in $N$. 
Since $A$ is abelian and $\displaystyle\phi\left(\frac{1}{2}\begin{pmatrix}1&1\\1&1\end{pmatrix}\right) \sim \begin{pmatrix}1&0\\0&0\end{pmatrix}$, 
there exist operators $a, b\in E(A)$ and $v\in \U(A)$ such that $a^2+b^2=1\in A$ and  
\[
\phi\left(\frac{1}{2}\begin{pmatrix}1&1\\1&1\end{pmatrix}\right) = 
\begin{pmatrix}a^2&abv\\abv^*&b^2\end{pmatrix}. 
\]
Considering $\operatorname{Ad}\begin{pmatrix}1&0\\0&v\end{pmatrix}\circ \phi$ instead of $\phi$, we may assume $v=1$. 
It follows that 
\[
\phi\left(\frac{1}{2}\begin{pmatrix}1&-1\\-1&1\end{pmatrix}\right) = 
\begin{pmatrix}b^2&-ab\\-ab&a^2\end{pmatrix}.
\]
By the assumption $\phi(1/2)=1/2$, we also have
\[
\phi\left(\frac{1}{4}\begin{pmatrix}1&1\\1&1\end{pmatrix}\right) = \frac{1}{2}
\begin{pmatrix}a^2&ab\\ab&b^2\end{pmatrix},\quad
\phi\left(\frac{1}{4}\begin{pmatrix}1&-1\\-1&1\end{pmatrix}\right) = \frac{1}{2}
\begin{pmatrix}b^2&-ab\\-ab&a^2\end{pmatrix}.
\]

We show $\phi(p/3)=p/3$. 
Since $\phi(p/3)\leq \phi(p)$, we can take $x\in E(A)$ such that $\phi(p/3)=\begin{pmatrix}x&0\\0&0\end{pmatrix}$. 
Then, by Lemma \ref{lem1}, we have $\phi(p/3 + p^{\perp})=\begin{pmatrix}x&0\\0&1\end{pmatrix}$. 
By Lemma \ref{lem4}, we obtain 
\[
\begin{pmatrix}1/3&0\\0&1\end{pmatrix} \geq \frac{1}{4}\begin{pmatrix}1&\pm1\\\pm1&1\end{pmatrix}
\]
in $M_2(A)$. 
Thus we have 
\begin{equation} \label{ineq}
\begin{pmatrix}x&0\\0&1\end{pmatrix} \geq\frac{1}{2}
\begin{pmatrix}a^2&ab\\ab&b^2\end{pmatrix}
\quad \text{and} \quad
\begin{pmatrix}x&0\\0&1\end{pmatrix} \geq\frac{1}{2}
\begin{pmatrix}b^2&-ab\\-ab&a^2\end{pmatrix}. 
\end{equation}
By Lemma \ref{lem4}, we obtain 
\[
2x\geq a^2b^2(2-b^2)^{-1} + a^2 \quad \text{and} \quad 
2x\geq a^2b^2(2-a^2)^{-1} + b^2.
\]
Since $A$ is abelian, we have $x\geq 1/3$ in $A$. That is, $\phi(p/3)\geq p/3$. 
The same discussion applied to $\phi^{-1}$ instead of $\phi$ shows that $\phi^{-1}(p/3)\geq p/3$ and thus $\phi(p/3)\leq p/3$. 
Hence we have $\phi(p/3) = p/3$. 
Furthermore, we have
$\phi(p/3+p^{\perp}) = p/3+p^{\perp}$.  
Similarly, we obtain $\phi(p^{\perp}/3) = \phi(p^{\perp})/3$. 
Note that $\{y\in E(M)\mid p/3\leq y\leq p/3+p^{\perp}\} = \{p/3 + y\mid y\in E(p^{\perp}Mp^{\perp}) \}$. 
Thus $1/3$ is the minimum element in $\{y\in E(M)\mid p/3\leq y\leq p/3+p^{\perp},\, y\geq p^{\perp}/3\}$.
It follows $\phi(1/3)=1/3$.  

A similar discussion shows that $\phi(2p/3) = 2p/3$ and $\phi(2/3)=2/3$. 
Considering the restriction $\phi|_{M_{[0, 2/3]}}\colon M_{[0, 2/3]}\to N_{[0, 2/3]}$ or $\phi|_{M_{[1/3, 1]}}\colon M_{[1/3, 1]}\to N_{[1/3, 1]}$ and iterative arguments, 
we have $\phi(cp)=cp$ and $\phi(c)=c$ for every $c$ in some dense subset of $[0, 1]$.  
Since $\phi$ is an order isomorphism, the same holds for every real number $0\leq c\leq 1$. 

By the inequality $(\ref{ineq})$ with $x=1/3$ and Lemma \ref{lem4}, we also have $a^2=b^2= 1/2$ in $A$ and thus 
\[
\phi\left(\frac{1}{2}\begin{pmatrix}1&1\\1&1\end{pmatrix}\right) = \frac{1}{2}\begin{pmatrix}1&1\\1&1\end{pmatrix}.
\] 
Using the same discussion and considering relative position among 
\[
p,\, \frac{1}{2}\begin{pmatrix}1&1\\1&1\end{pmatrix},\, \frac{1}{2}\begin{pmatrix}1&i\\-i&1\end{pmatrix}
\]
and their orthocomplements, we obtain 
\[
\phi\left( \frac{1}{2}\begin{pmatrix}1&i\\-i&1\end{pmatrix}\right) \in 
\left\{ \frac{1}{2}\begin{pmatrix}1&i\\-i&1\end{pmatrix},\,\frac{1}{2}\begin{pmatrix}1&-i\\i&1\end{pmatrix}\right\}.
\]
Note that, 
\begin{itemize}
\item (\cite[Lemma 1]{D}) for $e\in \P(M)$, $e\in\P(\Z(M))$ if and only if $\phi(e)\in\P(\Z(N))$. By Dye's theorem, $\phi|_{\P(\Z(M))}\colon \P(\Z(M))\to \P(\Z(N))$ extends to a (Jordan) $^*$-isomorphism from $\Z(M)\,\,(\cong A)$ onto $\Z(N)\,\,(\cong A)$.  
\item for a $^*$-automorphism $\pi\colon A\to A$, the mapping
\[
\begin{pmatrix}x_{11}&x_{12}\\x_{21}&x_{22}\end{pmatrix} \mapsto \begin{pmatrix}\pi(x_{11})&\pi(x_{12})\\\pi(x_{21})&\pi(x_{22})\end{pmatrix},\quad x_{11}, x_{12}, x_{21}, x_{22}\in A
\]
is a $^*$-automorphism on $M_2(A)$. 
\item the mapping 
\[
\begin{pmatrix}x_{11}&x_{12}\\x_{21}&x_{22}\end{pmatrix} \mapsto \begin{pmatrix}x_{11}&x_{21}\\x_{12}&x_{22}\end{pmatrix},\quad x_{11}, x_{12}, x_{21}, x_{22}\in A
\]
is a $^*$-antiautomorphism on $M_2(A)$ which fixes every element in $\Z(M_2(A))$. 
\end{itemize}
Therefore, we can take a Jordan $^*$-isomorphism $J\colon M\to N$ that satisfies $J(q)= \phi(q)$ for every $q$ in 
\[
\P(\Z(M)) \cup \left\{p, \frac{1}{2}\begin{pmatrix}1&1\\1&1\end{pmatrix},\, \frac{1}{2}\begin{pmatrix}1&i\\-i&1\end{pmatrix}\right\}. 
\]
We show that $J$ extends $\phi$. 
Equivalently, we show that $\psi:= J^{-1}\circ \phi$ is the identity mapping on $E(M)$. 

First, we prove 
\[
\psi\begin{pmatrix}s^2&st\\st&t^2\end{pmatrix}= \begin{pmatrix}s^2&st\\st&t^2\end{pmatrix}
\]
for every pair $(s, t)$ of real numbers with $s, t\geq 0$ and $s^2+t^2=1$. 
There exist operators $a_1, b_1\in E(A)$ and $v_1\in \U(A)$ such that $a_1^2+b_1^2 =1$ and 
\[
\psi\begin{pmatrix}s^2&st\\st&t^2\end{pmatrix}= \begin{pmatrix}a_1^2&a_1b_1v_1\\a_1b_1v_1^*&b_1^2\end{pmatrix}. 
\]
By Lemma \ref{lem4}, we obtain 
\[
\begin{split}
\begin{pmatrix}1/2&0\\0&1\end{pmatrix}
&\geq (2s^2 + t^2)^{-1}\begin{pmatrix}s^2&st\\st&t^2\end{pmatrix},\\
\begin{pmatrix}1/2&0\\0&1\end{pmatrix}
&\geq (s^2+ 2t^2)^{-1}\begin{pmatrix}t^2&-st\\-st&s^2\end{pmatrix}. 
\end{split}
\]
Since $\psi(ce) = c\psi(e)$ for every real number $0\leq c\leq 1$ and every projection $e\in \P(M)$, 
we have 
\[
\begin{split}
\begin{pmatrix}1/2&0\\0&1\end{pmatrix}
&\geq  (2s^2 + t^2)^{-1}\begin{pmatrix}a_1^2&a_1b_1v_1\\a_1b_1v_1^*&b_1^2\end{pmatrix},\\
\begin{pmatrix}1/2&0\\0&1\end{pmatrix}
&\geq (s^2+ 2t^2)^{-1}\begin{pmatrix}b_1^2&-a_1b_1v_1\\-a_1b_1v_1^*&a_1^2\end{pmatrix}. 
\end{split}
\]
Again by Lemma \ref{lem4}, we obtain $a_1=s$ and $b_1=t$. 

What we have just shown means that, for a maximal abelian projection $e$ in $M$ and a real number $0\leq c\leq 1$, we have $\lVert e-p\rVert = c$ and $\lVert e-p^{\perp}\rVert = (1-c^2)^{1/2}$ if and only if $\lVert \psi(e)-\psi(p)\rVert = c$ and $\lVert \psi(e)-\psi(p^{\perp})\rVert = (1-c^2)^{1/2}$. 
The same discussion shows 
\[
\begin{split}
\left\| \begin{pmatrix}s^2&st\\st&t^2\end{pmatrix}-\frac{1}{2}\begin{pmatrix}1&\pm1\\\pm1&1\end{pmatrix} \right\|
&= 
\left\| \psi\begin{pmatrix}s^2&st\\st&t^2\end{pmatrix}-\psi\left(\frac{1}{2}\begin{pmatrix}1&\pm1\\\pm1&1\end{pmatrix}\right) \right\|\\
&= 
\left\| \begin{pmatrix}s^2&stv_1\\stv_1^*&t^2\end{pmatrix}-\frac{1}{2}\begin{pmatrix}1&\pm1\\\pm1&1\end{pmatrix} \right\|
\end{split}
\]
and we also obtain $v_1=1$. 

Similarly, we can show 
\[
\psi\left(\frac{1}{2}\begin{pmatrix}1&\lambda\\ \bar{\lambda}&1\end{pmatrix}\right) = \frac{1}{2}\begin{pmatrix}1&\lambda\\ \bar{\lambda}&1\end{pmatrix}
\] 
for every $\lambda\in \mathbb{T}:=\{\gamma\in \mathbb{C}\mid \lvert \gamma\rvert = 1\}$. 
Using this, we can also prove that 
\[
\psi\begin{pmatrix}s^2&st\lambda\\st\bar{\lambda}&t^2\end{pmatrix} = \begin{pmatrix}s^2&st\lambda\\st\bar{\lambda}&t^2\end{pmatrix}
\] 
for any $\lambda\in \mathbb{T}$ and real numbers $s, t\geq 0$ with $s^2+t^2=1$. 

By the assumption on $J$, it follows that $\psi$ fixes operators of the form 
\[
\sum_{n=1}^{m} \left(c_n \begin{pmatrix}s_n^2&s_nt_n\lambda_n\\s_nt_n\overline{\lambda_n}&t_n^2\end{pmatrix} +  d_n \begin{pmatrix}t_n^2&-s_nt_n\lambda_n\\-s_nt_n\overline{\lambda_n}&s_n^2\end{pmatrix}\right) e_n, 
\]
where 
\begin{itemize}
\item $\{e_n\}_{n=1}^{m}\subset \P(\Z(M))$ is an orthogonal family of central projections in $M$, 
\item $c_n$, $d_n$ are real numbers in $[0, 1]$, $n=1,\ldots, m$,   
\item $\lambda_n\in \mathbb{T}$, $n=1,\ldots, m$, and 
\item real numbers $s_n, t_n\geq 0$ satisfy $s_n^2+t_n^2 = 1$, $n=1,\ldots, m$. 
\end{itemize} 
Since every element in $E(M)$ can be expressed as the supremum of some collection of such operators, $\psi$ fixes every element in $E(M)$. \smallskip

We consider the case $M$ and $N$ do not admit I$_2$ nor I$_1$ direct summands. 
By Dye's theorem, there exists a Jordan $^*$-isomorphism $J\colon M\to N$ that  satisfies $J(p)=\phi(p)$ for all $p\in\P(M)$. 
We show that $\psi:=J^{-1}\circ \phi$ is the identity mapping on $E(M)$. 
First we prove that $\psi(p/3)=p/3$ for every projection $p\in \P(M)$ with $p\prec p^{\perp}$. 
There exists a partial isometry $v\in M$ such that $vv^*=p$ and $q:=v^*v\leq p^{\perp}$. 
Then we can identify $y\in (p+q)M(p+q)$ with 
$\begin{pmatrix}pyp&pyv^*\\vyp&vyv^*\end{pmatrix}\in M_2(pMp)$. 
We can take $x\in E(pMp)$ that satisfies 
\[
\psi(p/3) = \psi\begin{pmatrix}1/3&0\\0&0\end{pmatrix} = \begin{pmatrix}x&0\\0&0\end{pmatrix}. 
\]
It follows by Lemma \ref{lem1} that 
\[
\psi(p/3 + q) = \psi\begin{pmatrix}1/3&0\\0&1\end{pmatrix} = \begin{pmatrix}x&0\\0&1\end{pmatrix}. 
\]
Since 
\[
\begin{pmatrix}1/3&0\\0&1\end{pmatrix} \geq \frac{1}{4}\begin{pmatrix}1&1\\1&1\end{pmatrix}, 
\]
we have 
\[
\begin{pmatrix}x&0\\0&1\end{pmatrix} \geq\psi\left(\frac{1}{4}
\begin{pmatrix}1&1\\1&1\end{pmatrix}\right) = \frac{1}{4}
\begin{pmatrix}1&1\\1&1\end{pmatrix}.
\]
By Lemma \ref{lem4}, we obtain $x\geq 1/3$ in $pMp$. 
Use the same discussion on $\psi^{-1}$ to obtain $\psi(p/3)=p/3$. 

We can take projections $p_1, p_2, p_3\in\P(M)$ such that $p_1+p_2+p_3 = 1$, $p_k\prec p_k^{\perp}$, $k=1, 2, 3$. 
Then $\psi(p_k/3) = p_k/3$. 
By Lemma \ref{lem1}, we also have  $\psi(p_1/3 + p_1^{\perp}) = p_1/3 + p_1^{\perp}$. 
By the fact $(p_1+ p_2)/3$ is the minimum element in $\{y\in M_{[p_1/3, p_1/3 + p_1^{\perp}]}\mid y\geq p_2/3\}$, we have $\psi((p_1+ p_2)/3) = (p_1+ p_2)/3$ and additionally $\psi((p_1+ p_2)/3 + p_3) = (p_1+ p_2)/3 + p_3$. 
By the fact $1/3$ is the minimum element in $\{y\in M_{[(p_1+p_2)/3, (p_1+p_2)/3+ p_3]}\mid y\geq p_3/3\}$, we have $\psi(1/3) = 1/3$. 
By the discussion as in the first half of this proof, we have $\psi(c)=c$ for every $c\in [0, 1]$. 

We show by induction that for any $m\in \mathbb{N}$, $c_1, \ldots, c_m\in [0, 1]$ and  mutually orthogonal $q_1, \ldots, q_m\in \P(M)$, we have 
\[
\psi\left(\sum_{n=1}^m c_nq_n\right)= \sum_{n=1}^{m} c_nq_n. 
\]
The fact that $c_1q_1$ is the maximum element in the set $\{y\in E(M)\mid y\leq c_1,\,\, y\leq q_1\}$ shows $\psi(c_1q_1) = c_1q_1$. 
Suppose that $\psi(\sum_{n=1}^{m-1} c_nq_n)= \sum_{n=1}^{m-1} c_nq_n$ holds. 
By Lemma \ref{lem1}, we also have $\psi(\sum_{n=1}^{m-1} c_nq_n + (1-\sum_{n=1}^{m-1} q_n))= \sum_{n=1}^{m-1} c_nq_n + (1-\sum_{n=1}^{m-1} q_n)$. 
By the fact that $\sum_{n=1}^{m} c_nq_n$ is the minimum element in 
$\{y\in M_{[(\sum_{n=1}^{m-1} c_nq_n, \sum_{n=1}^{m-1} c_nq_n + (1-\sum_{n=1}^{m-1} q_n)]}\mid y\geq c_mq_m\}$, 
we obtain $\psi(\sum_{n=1}^{m} c_nq_n)= \sum_{n=1}^{m} c_nq_n$. 
Since $\psi$ is an order automorphism, $\psi$ fixes every element in $E(M)$. 
\end{proof}

We would like to consider general order isomorphisms between effect algebras. 
In the case of $\B(H)$, \v{S}emrl obtained the following description. 
\begin{theorem}[\v{S}emrl, {\cite[Theorem 2.2]{S17}}]
Assume $\operatorname{dim}H\geq 2$ and $\phi\colon E(\B(H))\to E(\B(H))$ is a mapping. 
Then the following two conditions are equivalent: 
\begin{enumerate}[$(1)$]
\item The mapping $\phi$ is an order isomorphism.  
\item There exist real numbers $\alpha, \beta$ with $0<\alpha<1$, $\beta<0$ and a bounded bijective  linear or conjugate-linear operator $T\colon H\to H$ with $\lVert T\rVert \leq 1$ such that 
\[
\phi(a) = f_{\beta}\left( (f_{\alpha}(TT^*))^{-1/2} f_{\alpha}(TaT^*) (f_{\alpha}(TT^*))^{-1/2}\right)
\]
for every $a\in E(\B(H))$. 
\end{enumerate}
\end{theorem}

The following slightly modified version holds in general settings of von Neumann algebras. 

\begin{proposition}\label{pq}
Let $M$ and $N$ be von Neumann algebras without type I$_1$ direct summands.  
Assume that $\phi\colon E(M)\to E(N)$ is a mapping. 
Then the following two conditions are equivalent: 
\begin{enumerate}[$(1)$]
\item The mapping $\phi$ is an order isomorphism such that $\phi(1/2)$ is invertible in $N$.  
\item There exist a Jordan $^*$-isomorphism $J\colon M\to N$ and real numbers $\alpha, \beta$ with $0<\alpha<1$, $\beta<0$ and a locally measurable positive operator $T\in \LS(N)$ which is invertible in $\LS(N)$ such that 
\[
\phi(a) = f_{\beta}\left( (f_{\alpha}(T^2))^{-1/2} f_{\alpha}(TJ(a)T) (f_{\alpha}(T^2))^{-1/2}\right)
\]
for every $a\in E(M)$. 
Here we consider operations in $\LS(N)$. 
\end{enumerate}
Moreover, if $M$ and $N$ are Jordan $^*$-isomorphic and if $b_0\in E(N)$ is invertible in $N$ and $1-b_0$ is locally measurably invertible in $N$, then there exists an order isomorphism $\phi\colon E(M)\to E(N)$ as in $(2)$ with $\phi(1/2) =b_0$. 
\end{proposition}
\begin{proof}
Recall that the function
\[
f_{\alpha}(t)= \frac{t}{\alpha t+1-\alpha},\quad t\in\mathbb{R}_+
\]
is operator monotone.
From this, it is not difficult to deduce that, if $\phi$ is as in $(2)$, then $\phi$ is an order isomorphism from $E(M)$ onto $E(N)$. 
In $(2)$, we have $\phi(1/2) = f_{\beta}\left( (f_{\alpha}(T^2))^{-1/2} f_{\alpha}(T^2/2) (f_{\alpha}(T^2))^{-1/2}\right) = F_{\alpha, \beta}(T)$, where 
\[
F_{\alpha, \beta}(t) = \frac{\alpha t^2+1-\alpha}{\alpha t^2+(2-\beta)(1-\alpha)},\quad t\geq 0. 
\]
Note that $F_{\alpha, \beta}$ is strictly increasing and $F_{\alpha, \beta}([0, \infty)) = [(2-\beta)^{-1}, 1)$. 
Hence $\phi(1/2)$ is invertible in $N$. 

Suppose conversely that $\phi\colon E(M)\to E(N)$ is an order isomorphism and $\phi(1/2)$ is invertible in $N$. 
Take real numbers $\alpha, \beta$ with $0<\alpha<1$, $\beta<0$, $(2-\beta)^{-1}\leq \phi(1/2)$. 
Put $T := F_{\alpha, \beta}^{-1}(\phi(1/2))$. 
Note that $1-\phi(1/2)$ is locally measurably invertible in $N$ and thus $T$ is well defined. 
Moreover, $T$ is a locally measurable positive unbounded operator which is invertible in $\LS(N)$ and satisfies $F_{\alpha, \beta} (T) = \phi(1/2)$. 
Define an order automorphism $\psi$ on $E(N)$ by 
\[
\psi(b) = f_{\beta}\left( (f_{\alpha}(T^2))^{-1/2} f_{\alpha}(TbT) (f_{\alpha}(T^2))^{-1/2}\right),\quad b\in E(N).
\]
Then the order isomorphism $\psi^{-1}\circ\phi\colon E(M)\to E(N)$ satisfies $\psi^{-1}\circ \phi(1/2) = 1/2$. 
By Theorem \ref{Jordan}, $\psi^{-1}\circ\phi$ extends to a Jordan $^*$-isomorphism $J\colon M\to N$, and we have $\phi = \psi\circ J$ on $E(M)$. 

The same discussion also shows the last assertion of this proposition.
\end{proof}

\begin{remark}
Parallel to \cite[Theorem 2.3]{S17}, we can also prove that the conditions $(1)$ and $(2)$ in Proposition \ref{pq} are equivalent to:
\begin{enumerate}[$(1)$]
\item[(3)] There exist a Jordan $^*$-isomorphism $J\colon M\to N$, a negative real number $\alpha<0$ and a locally measurable positive operator $T\in \LS(N)$ which is invertible in $\LS(N)$ such that 
\[
\phi(a) = f_{\alpha} \left( (1+T^{-2})^{1/2} (1-(1+TJ(a)T)^{-1}) (1+T^{-2})^{1/2} \right)
\]
for every $a\in E(M)$. 
\end{enumerate}
We omit the proof because it can be done in a similar way. 
We do not use this fact in the rest of this paper. 
\end{remark}

Proposition \ref{pq} shows the following.

\begin{proposition}\label{characterization}
Let $M$ and $N$ be two von Neumann algebras and suppose $b\in E(N)$. 
If $M$ and $N$ are Jordan $^*$-isomorphic, then the following two conditions are equivalent: 
\begin{enumerate}[$(1)$]
\item Both $b$ and $1-b$ are locally measurably invertible in $N$.  
\item There exists an order isomorphism $\phi\colon E(M)\to E(N)$ such that $\phi(1/2) = b$. 
\end{enumerate}
\end{proposition}
\begin{proof}
By Lemma \ref{invertible}, it is easy to see that $(2)\Rightarrow(1)$ holds. 

$(1)\Rightarrow(2)$ 
By Proposition \ref{pq}, there exists an order automorphism $\phi_1\colon E(N)\to E(N)$ that satisfies $\phi_1(1/2)\geq b$. 
Put $b_0:= \phi_1^{-1}(b)$. 
Then $b_0$ is a locally measurably invertible element in $N$ with $b_0\leq 1/2$.  Again by Proposition \ref{pq}, there exists an order isomorphism $\phi_2\colon E(M)\to E(N)$ with $\phi_2(1/2) = 1-b_0$. 
Then the order isomorphism $\phi\colon E(M)\to E(N)$ defined by $\phi(a) =\phi_1(1-\phi_2(1-a))$, $a\in E(M)$ satisfies the desired condition. 
\end{proof}

In type I$_1$ cases, we can completely understand the general form of order isomorphisms between effect algebras essentially by the measure theoretic idea of Cabello S\'anchez in \cite{CS}. 
Suppose that $\phi\colon E(M)\to E(N)$ is an order isomorphism between effect algebras of commutative von Neumann algebras. 
Then $\phi|_{\P(M)}$ is an orthoisomorphism from $\P(M)$ onto $\P(N)$. 
By Dye's theorem, there exists a Jordan $^*$-isomorphism (which is actually a $^*$-isomorphism because $M$ and $N$ are commutative) that extends $\phi|_{\P(M)}$. 

Thus it suffices to give, for a commutative von Neumann algebra $A$, the general form of order automorphisms on $E(A)$ that fix every projection. 
We can take some measure space $(X, \mu)$ such that $A$ is $^*$-isomorphic to $L^{\infty}(\mu)$. 
Hence the following proposition gives a complete characterization. 

\begin{proposition}\label{abel}
Let $(X, \mu)$ be a measure space. 
Put $E(L^{\infty}(\mu)):=\{f\in L^{\infty}(\mu)\mid 0\leq f\leq 1\}$ and $\P(L^{\infty}(\mu)):=\{p\in L^{\infty}(\mu)\mid p(x)\in \{0,\, 1\}\text{ a.e. }x\in X\}$. 
Suppose that $\phi\colon E(L^{\infty}(\mu))\to E(L^{\infty}(\mu))$ is an order isomorphism that satisfies $\phi(p)=p$ for every $p\in \P(L^{\infty}(\mu))$. 
Then there exists a measurable mapping $\tau\colon X\times [0, 1] \to [0, 1]$ that satisfies the following conditions. 
\begin{enumerate}[$(1)$]
\item For every $x\in X$, the mapping $\tau(x, \cdot)$ is a continuous monotone increasing bijection on $[0, 1]$. 
\item For every $f\in E(L^{\infty}(\mu))$,  we have $(\phi(f))(x) = \tau(x, f(x))$ a.e.\ $x\in X$. 
\end{enumerate}
Conversely, if a measurable mapping $\tau\colon X\times [0, 1] \to [0, 1]$ satisfies $(1)$, 
then the mapping $\phi\colon E(L^{\infty}(\mu))\to E(L^{\infty}(\mu))$ that is determined by the equation in $(2)$ is an order automorphism on $E(L^{\infty}(\mu))$. 
\end{proposition}
\begin{proof}
The first half can be proved by the same method as \cite[Lemma 9]{C} so we omit its proof. 
Suppose that a measurable function $\tau\colon X\times [0, 1] \to [0, 1]$ satisfies the condition $(1)$. 
Define $\tilde{\tau}\colon X\times [0, 1] \to [0, 1]$ by 
$\tilde{\tau} (x, c) = \tau(x, \cdot)^{-1}(c)$, $(x, c)\in X\times [0, 1]$. 
Then $\tilde{\tau}$ is a measurable mapping. 
Indeed, for every $d\in \mathbb{R}$, the set
\[
\begin{split}
\{(x, c)\in X\times [0, 1]\mid \tilde{\tau}(x, c) > d\} 
&= \{(x, c)\in X\times [0, 1]\mid \tau(x, d)< c\}\\
&= \bigcup_{c_0\in \mathbb{Q}}\{(x, c)\in X\times [0, 1]\mid \tau(x, d) <c_0\leq c\}
\end{split}
\]
is a measurable set in $X\times [0, 1]$.
Thus the second half follows by the same discussion as in \cite[Theorem 3]{C}.  
\end{proof}

\section{Order isomorphisms between operator intervals}\label{interval}
In this section, we use the results in the preceding section to investigate order isomorphisms between general operator intervals in von Neumann algebras. 
The idea of this section is motivated by \cite[Theorem 3.1]{S18} due to \v{S}emrl.

For a von Neumann algebra $M$, one can show that 
the mapping $x\mapsto (1-x)^{-1} -1$ is an order isomorphism 
from the collection $\{x\in E(M)\mid 1-x \text{ is invertible in }\LS(M)\}$ onto $\LS(M)_+ := \{\text{positive elements in }\LS(M)\}$.
The inverse mapping is $X\mapsto 1-(1+X)^{-1}$, $X\in \LS(M)_+$. 

\begin{proposition}\label{linear}
Let $M$ and $N$ be von Neumann algebras without type I$_1$ direct summands. 
Suppose that $\phi\colon E(M)\to E(N)$ is an order isomorphism. 
Define an order isomorphism $\Phi\colon \LS(M)_+\to \LS(N)_+$ by 
\[
\Phi(X) := \left(1-\phi(1- (1+X)^{-1})\right)^{-1} -1,\quad X\in \LS(M)_+. 
\]
Then there exist a Jordan $^*$-isomorphism $J\colon M\to N$ and an element $B\in \LS(N)_+$ which is invertible in $\LS(N)$ such that 
\[
\Phi(X) = BJ(X)B,\quad X\in  \LS(M)_+. 
\]
(Note that $J$ can be canonically extended to a bijection from $\LS(M)$ onto $\LS(N)$ and this extension preserves positivity, Jordan products, inverses, etc.)

In other words, 
for every $x\in E(M)$ with $1-x$ invertible in $\LS(M)$, 
putting $X= (1-x)^{-1} -1\in \LS(M)_+$, we have
\[
\phi(x) = 1-(\Phi(X)+1)^{-1} = 1- \left(BJ((1-x)^{-1} -1)B+1\right)^{-1}. 
\]
\end{proposition}
\begin{proof}
Let $M_i$, $i=1, 2, 3$ be von Neumann algebras, $J_1\colon M_1\to M_2$, $J_2\colon M_2\to M_3$ be Jordan $^*$-isomorphisms and $B_i\in \LS(M_i)_+$ be invertible in $\LS(M_i)$, $i=2, 3$. 
Suppose that  the order isomorphisms $\Phi_i\colon \LS(M_i)_+\to \LS(M_{i+1})_+$, $i=1, 2$ are defined by
\[
\Phi_i(X) = B_{i+1}J_i(X)B_{i+1},\quad X\in  \LS(M_i)_+. 
\] 
Then, for $X\in  \LS(M_1)_+$, we have 
\[
\begin{split}
\Phi_2\circ \Phi_1(X) &= \Phi_2(B_2J_1(X)B_2)\\
&= B_3 J_2(B_2J_1(X)B_2) B_3\\
&= B_3 J_2(B_2) J_2(J_1 (X)) J_2(B_2) B_3\\
&= \lvert J_2(B_2)B_3 \rvert u^* J_2( J_1 (X)) u \lvert J_2(B_2)B_3 \rvert,
\end{split}
\]
where $J_2(B_2)B_3 = u \lvert J_2(B_2)B_3\rvert $ is the polar decomposition. 
Then $u\in \U(M_3)$ and thus the mapping $x\mapsto u^* J_2( J_1 (x)) u$ is a Jordan $^*$-isomorphism from $M_1$ onto $M_3$. 
In addition, $\lvert J_2(B_2)B_3 \rvert\in \LS(M_3)_+$ is invertible in $\LS(M_3)$.  

Therefore, by the discussion in Propositions \ref{pq} and \ref{characterization}, it suffices to consider the case
\begin{itemize}
\item $\phi = J|_{E(M)}$ for some Jordan $^*$-isomorphism $J\colon M\to N$, 
\item $M=N$ and $\phi(a) =f_{\beta}(a)$ for some real number $\beta<0$, 
\item $M=N$ and $\phi(a) = (f_{\alpha}(T^2))^{-1/2} f_{\alpha}(TaT) (f_{\alpha}(T^2))^{-1/2}$ for some $0<\alpha<1$ and a locally measurable positive operator $T\in \LS(N)$ that is invertible in $\LS(N)$, and
\item $M=N$ and $\phi$ is defined by $a\mapsto 1-\psi(1-a)$, $a\in E(M)$ when $\psi$ satisfies the conclusion of this proposition.
\end{itemize}

In the first case, we have 
\[
\begin{split}
\Phi(X) &= \left(1-J(1- (1+X)^{-1})\right)^{-1} -1\\
&= J\left((1-(1- (1+X)^{-1})\right)^{-1} -1\\
&= J(1+X) -1 = J(X)
\end{split}
\]
for every $X\in  \LS(M)_+$. 

In the second case, we obtain  
\[
\begin{split}
\Phi(X) &= \left(1-f_{\beta}(1- (1+X)^{-1})\right)^{-1} -1\\
&= \left(1-X(X + (1-\beta))^{-1}\right)^{-1} -1\\
&= \left((1-\beta)(X + (1-\beta))^{-1}\right)^{-1} -1 = (1-\beta)^{-1}X 
\end{split}
\]
for every $X\in  \LS(M)_+$. 

The third case is a little complicated. 
Let $X\in  \LS(M)_+$ be an invertible element in $\LS(M)$. 
By the equality
\[
\Phi(X) = \left(1-\phi(1- (1+X)^{-1})\right)^{-1} -1, 
\]
we have 
\[
\Phi(X)(\Phi(X)+1)^{-1} = \phi(1- (1+X)^{-1}).
\]
Take inverses of both sides to obtain
\[
\begin{split}
&\quad 1 + \Phi(X)^{-1}\\
&= \phi(1- (1+X)^{-1})^{-1}\\
&= f_{\alpha}(T^2)^{1/2} f_{\alpha}\left(T(1- (1+X)^{-1})T\right)^{-1} f_{\alpha}(T^2)^{1/2}\\
&= f_{\alpha}(T^2)^{1/2} \left(\alpha + (1-\alpha)\left(T(1- (1+X)^{-1})T\right)^{-1}\right) f_{\alpha}(T^2)^{1/2}\\
&= \alpha f_{\alpha}(T^2) + (1-\alpha) f_{\alpha}(T^2)^{1/2} T^{-1} (X^{-1} + 1) T^{-1} f_{\alpha}(T^2)^{1/2}\\
&= \alpha f_{\alpha}(T^2) + (1-\alpha) T^{-2} f_{\alpha}(T^2) + (1-\alpha) f_{\alpha}(T^2)^{1/2} T^{-1} X^{-1} T^{-1} f_{\alpha}(T^2)^{1/2}\\
&= f_{\alpha}(T^2)^{-1} f_{\alpha}(T^2) + (1-\alpha) f_{\alpha}(T^2)^{1/2} T^{-1} X^{-1} T^{-1} f_{\alpha}(T^2)^{1/2}\\
&= 1+ (1-\alpha) f_{\alpha}(T^2)^{1/2} T^{-1} X^{-1} T^{-1} f_{\alpha}(T^2)^{1/2}.
\end{split}
\]
Subtracting $1$ from both sides and taking inverses, we have
\[
\Phi(X) = \frac{1}{1-\alpha}  f_{\alpha}(T^2)^{-1/2} TXT f_{\alpha}(T^2)^{-1/2}. 
\]
Since $\Phi$ is an order isomorphism, the same equality holds for all $X\in  \LS(M)_+$. 

We consider the fourth case. 
Thus we suppose that $\psi$ satisfies the conclusion. 
Define $\Psi\colon \LS(M)_+\to \LS(M)_+$ by 
\[
\Psi(X) := \left(1-\psi(1- (1+X)^{-1})\right)^{-1} -1,\quad X\in \LS(M)_+. 
\]
Take $J_0$ and $B_0$ such that $\Psi(X) = B_0J_0(X)B_0$ for any  $X\in \LS(M)_+$. 
Consider the following composition of mappings for $X\in \LS(M)_+$ that is invertible in $\LS(M)$. 
\[
\begin{split}
X&\mapsto 1-(1+X)^{-1}\\
&\mapsto 1-(1-(1+X)^{-1}) = (1+X)^{-1}\\
&\mapsto (1-(1+X)^{-1})^{-1}-1 = X^{-1}=: \omega(X). 
\end{split}
\]
Then 
\[
\Phi (X) = {\omega}^{-1}\circ \Psi \circ \omega (X) = {\omega}^{-1}\circ \Psi (X^{-1}) = {\omega}^{-1} (B_0J_0(X^{-1})B_0) = B_0^{-1} J_0(X) B_0^{-1}
\]
for such an $X$.
Since $\Phi$ is an order isomorphism, the same equation holds for every $X\in \LS(M)_+$
\end{proof}

\begin{corollary}\label{affine}
Let $M$ and $N$ be two von Neumann algebras without type I$_1$ direct summands. 
Suppose that $K\subset M_{sa}$ and $L\subset N_{sa}$ are operator intervals and that $\phi\colon K \to L$ is an order isomorphism. 
Suppose further that there exist two operators $a_1, a_2\in K$ such that 
$a_2-a_1$ is locally measurably invertible in $M$ and $\phi((a_1+a_2)/2) = (\phi(a_1) + \phi(a_2))/2$. 
Then $\phi$ is an affine mapping on $K$. 
\end{corollary}
\begin{proof}
It suffices to show that, 
for every $a_0, a\in K$ with $a_0\leq a_1$ and $a_2\leq a$, $\phi$ is affine on the operator interval $M_{[a_0, a]}$. 
In the rest of this proof, we show that $\phi$ is affine on $M_{[a_1, a]}$. 
(Then a similar discussion shows that $\phi$ is affine on $M_{[a_0, a]}$.)

By normalization as in the proof of Lemma \ref{normalize}(3), we may assume that $a_1=0, a = 1\in M$ and $\phi(0)=0, \phi(1) = 1 \in N$. 
By Theorem \ref{Jordan}, $\phi$ is affine on $M_{[0, a_2]}$. 
Take an element $a_3\in E(M)$ with the property that $a_3 \leq a_2$ and both $a_3$ and $1-a_3$ are invertible in $\LS(M)$. 
Then $\phi(a_3/2)= \phi(a_3)/2$. 
By the preceding proposition, there exist a Jordan $^*$-isomorphism $J\colon M\to N$ and an element $B\in \LS(N)_+$ which is invertible in $\LS(N)$ such that
\[
\phi(x) = 1- \left(BJ((1-x)^{-1} -1)B+1\right)^{-1}
\]
for every $x\in E(M)$ with $1-x$ invertible in $\LS(M)$. 
By the equation $\phi(a_3/2)= \phi(a_3)/2$, we have 
\[
1- \left(BJ((1-a_3/2)^{-1} -1)B+1\right)^{-1} = \frac{1}{2}\left(1- \left(BJ((1-a_3)^{-1} -1)B+1\right)^{-1}\right)
\]
Take inverses of both sides. Then
\[
1+ B^{-1} J((1-a_3/2)^{-1} -1)^{-1}B^{-1} = 2 (1+ B^{-1} J((1-a_3)^{-1} -1)^{-1}B^{-1}). 
\]
We can easily see 
\[
J((1-a_3/2)^{-1} -1)^{-1} = J(2a_3^{-1} -1) = 2J(a_3)^{-1} -1
\]
and
\[
J((1-a_3)^{-1} -1)^{-1} = J(a_3^{-1} -1) = J(a_3)^{-1} -1. 
\]
By these equalities, we obtain 
\[
1+ B^{-1}  (2J(a_3)^{-1} -1)B^{-1} = 2 (1+ B^{-1} (J(a_3)^{-1} -1)B^{-1}). 
\]
Subtracting $1$ and multiplying $B$ from both sides, we have 
\[
2J(a_3)^{-1} -1 = B^2 +2(J(a_3)^{-1} -1). 
\]
Thus $B^2=1$ and since $B$ is positive we finally obtain $B=1$. 
It follows that 
\[
\phi(x) = 1- \left(J((1-x)^{-1} -1)+1\right)^{-1} = J(x)
\]
for every $x\in E(M)$ with $1-x$ invertible in $\LS(M)$. 
Since $\phi$ is an order isomorphism, the same equation holds for every $x\in E(M)$.
\end{proof}

We are ready to consider general order isomorphisms between operator intervals. 
By the following theorem combined with Propositions \ref{pq},  \ref{characterization} and arguments in Section \ref{pre1}, we can completely understand the general form of order isomorphisms between operator intervals in the setting of von Neumann algebras. 

\begin{theorem}\label{rest}
Let $M$ and $N$ be von Neumann algebras without type I$_1$ direct summands. 
\begin{enumerate}[$(1)$]
\item If $\phi \colon M_{sa}\to N_{sa}$  is an order isomorphism, then there exist a Jordan $^*$-isomorphism $J\colon M\to N$ and elements $x\in N^{-1}$, $b\in N_{sa}$ that satisfy $\phi(a) = xJ(a)x^* + b$ for every $a\in M_{sa}$. 
\item If $\phi \colon M_+\to N_+$ or  $M_{(-\infty, 0]}\to N_{(-\infty, 0]}$ or $M_{(0, \infty)}\to N_{(0, \infty)}$ is an order isomorphism, then there exist a Jordan $^*$-isomorphism $J\colon M\to N$ and an element $x\in N^{-1}$ that satisfy $\phi(a) = xJ(a)x^*$ for every $a$ in the domain of $\phi$. 
\item Two operator intervals $M_{(0, \infty)}$ and $N_{sa}$ can never be order isomorphic. 
\end{enumerate}
\end{theorem}
\begin{proof}
We first consider the case $\phi\colon M_+\to N_+$ is an order isomorphism. 
We show that $\phi(2)= 2\phi(1)$. 
Put $\phi(1) =: b_1$ and $\phi(2) =: b_2$. 
Let $c>2$ be a positive real number. 
We would like to express $\phi(c)$ by means of $b_1$ and $b_2$. 
We know that $\phi$ restricts to an order isomorphism from $M_{[0, c]}$ onto $N_{[0, \phi(c)]}$. 
By Proposition \ref{linear}, the mapping $\Psi_c\colon \LS(M)_+\to \LS(N)_+$ that is defined by 
\[
\Psi_c(X) = \left(1- \phi(c)^{-1/2}\phi(c(1-(1+X)^{-1}))\phi(c)^{-1/2} \right)^{-1} -1,\quad X\in \LS(M)_+
\]
is a linear mapping. 
We can compute: 
\[
\begin{split}
\Psi_c\left(\frac{1}{c-1}\right) &= (1-\phi(c)^{-1/2}b_1\phi(c)^{-1/2})^{-1} -1,\\
\Psi_c\left(\frac{2}{c-2}\right) &= (1-\phi(c)^{-1/2}b_2\phi(c)^{-1/2})^{-1} -1. 
\end{split}
\]
It follows by linearity that
\[
\frac{1}{c-1} \left((1-\phi(c)^{-1/2}b_2\phi(c)^{-1/2})^{-1} -1\right) =
\frac{2}{c-2} \left((1-\phi(c)^{-1/2}b_1\phi(c)^{-1/2})^{-1} -1\right).
\]
Taking inverses in $\LS(N)$, we have
\[
(c-1) ((\phi(c)^{-1/2}b_2\phi(c)^{-1/2})^{-1} -1) = \frac{c-2}{2}((\phi(c)^{-1/2}b_1\phi(c)^{-1/2})^{-1} -1). 
\]
Consider multiplication by $\phi(c)^{-1/2}\in\LS(N)$ from both sides to obtain
\[
(c-1)b_2^{-1} - (c-1)\phi(c)^{-1} = \frac{c-2}{2} (b_1^{-1} - \phi(c)^{-1}). 
\]
We come to the equation 
\[
\phi(c)^{-1} = \left(2-\frac{2}{c}\right) b_2^{-1} - \left(1- \frac{2}{c}\right) b_1^{-1}. 
\]
Since $\phi(c)^{-1} \to 0$ in norm as $c\to\infty$, we have 
\[
0 = \lim_{c\to \infty} \left(\left(2-\frac{2}{c}\right) b_2^{-1} - \left(1- \frac{2}{c}\right) b_1^{-1} \right)= 2b_2^{-1}-b_1^{-1}
\]
and thus we obtain $b_2 = 2b_1$. 
That is, $\phi(2) = 2\phi(1)$. 

By Corollary \ref{affine}, $\phi$ is affine on $M_+$. 
Thus $\phi(1)$ is invertible in $N$. 
By Theorem \ref{Jordan} and by affinity, the mapping $a\mapsto \phi(1)^{-1/2}\phi(a)\phi(1)^{-1/2}$, $a\in M_+$ extends to a Jordan $^*$-isomorphism from $M$ onto $N$. 
Hence the case of the pair $(M_+, N_+)$ is completed. 
A similar discussion also completes the case of the pair $(M_{(-\infty, 0]}, N_{(-\infty, 0]})$. 

If $\phi$ is an order isomorphism from $M_{sa}$ onto $N_{sa}$, consider the restriction $\phi|_{M_+}$ which is an order isomorphism from $M_+$ onto $N_{[\phi(0), \infty)}$. 
The same discussion as above shows that $\phi$ is affine on $M_+$, and we see that an equation as in the statement of $(1)$ is valid on $M_+$. 
By Corollary \ref{affine}, $\phi$ is affine on $M_{sa}$ and thus the equation holds for all $a\in M_{sa}$. 

Assume that $\phi$ is an order isomorphism from $M_{(0, \infty)}$ onto $N_{(0, \infty)}$ or from $M_{(0, \infty)}$ onto $N_{sa}$. 
Consider the restriction $\phi|_{M_{[1, \infty)}}$ which is an order isomorphism from $M_{[1, \infty)}$ onto $N_{[\phi(1), \infty)}$. 
The same discussion shows that this restriction is affine, and by Corollary \ref{affine}, $\phi$ is affine on $M_{(0, \infty)}$. 
This shows $(3)$. 
If $\phi$ is an order isomorphism from $M_{(0, \infty)}$ onto $N_{(0, \infty)}$, 
$\phi$ can be expressed as the equation in $(1)$, 
but the condition $\phi(M_{(0, \infty)}) = N_{(0, \infty)}$ implies $b=0$. 
\end{proof}

Type I$_1$ cases can be characterized by a measure theoretic method as in Proposition \ref{abel}. 
We omit the details. 

We remark that, for a unital commutative C$^*$-algebra $A$, $A_{sa}$ is order isomorphic to $A_{(0,\infty)}:=\{a\in A\mid 0\leq a,\,\, a \text{ is invertible}\}$. 
Indeed, we can take an order isomorphism $f\colon \mathbb{R}\to (0, \infty)$. Then the mapping $\phi\colon A_{sa}\to A_{(0, \infty)}$ that is defined by $\phi(a) = f(a)$, $a\in A_{sa}$, is an order isomorphism.\medskip 

Lastly, we give another characterization of order isomorphisms between effect algebras. 

\begin{proposition}\label{last}
Let $M$ and $N$ be von Neumann algebras without type I$_1$ direct summands. 
Suppose that $M$ is Jordan $^*$-isomorphic to $N$. 
As in the statement of Proposition \ref{linear}, for an order isomorphism $\phi\colon E(M)\to E(N)$,
we define an order isomorphism $\Phi\colon \LS(M)_+\to \LS(N)_+$ by 
\[
\Phi(X) := \left(1-\phi(1- (1+X)^{-1})\right)^{-1} -1,\quad X\in \LS(M)_+. 
\]
Then the mapping $\phi\mapsto\Phi$ is a bijection from
\[
\{\phi\colon E(M) \to E(N) \text{ order isomorphisms}\}
\]
onto
\[
\{\Phi\colon \LS(M)_+\to \LS(N)_+ \text{ order isomorphisms}\}. 
\]
\end{proposition}
\begin{proof}
It is easy to see that the mapping $\phi\mapsto \Phi$ is injective. 
Hence it suffices to show surjectivity. 
Let $\Phi\colon\LS(M)_+\to \LS(N)_+$ be an order isomorphism. 
By imitating the discussion in the proof of Theorem \ref{rest}, we can show the following: 
There exist a Jordan $^*$-isomorphism $J\colon M\to N$ and an element $B\in \LS(N)_+$ which is invertible in $\LS(N)$ such that 
\[
\Phi(X) = BJ(X)B,\quad X\in  \LS(M)_+.  
\]
Since $B\in \LS(N)_+$ is invertible in $\LS(N)$, 
both $1-(1+B^2)^{-1}$ and $1-(1-(1+B^2)^{-1})\,\,(=(1+B^2)^{-1})$ are locally measurably invertible in $N$. 
By Proposition \ref{characterization}, there exists an order isomorphism $\phi_0\colon E(M)\to E(N)$ such that $\phi_0(1/2) = 1-(1+B^2)^{-1}$. 
Then the corresponding mapping $\Phi_0\colon \LS(M)_+\to \LS(N)_+$ satisfies $\Phi_0(1) = B^2$. 
By Proposition \ref{linear}, there exists a Jordan $^*$-isomorphism $J_0\colon M\to N$ such that 
$\Phi_0(X) = BJ_0(X)B$ for every $\LS(M)_+$. 
Put $J_0^{-1}\circ J =: J_1\colon M\to M$. 
By the proof of Proposition \ref{linear}, $\phi_1:= J_1|_{E(M)} \colon E(M)\to E(M)$ corresponds to the order isomorphism $\Phi_1= J_1\colon \LS(M)_+\to \LS(M)_+$. 
Thus the order isomorphism $\phi_0\circ \phi_1\colon E(M)\to E(N)$ corresponds to $\Phi_0\circ \Phi_1 = \Phi$. 
\end{proof}

It is easy to see that, if in addition we assume $M=N$ in this proposition, then the mapping $\phi\mapsto \Phi$ is actually a group isomorphism between the groups of order automorphisms.

\section{Problems}\label{problem}
Our results give rise to further problems on order isomorphisms of operator algebras. 
In this section, we list some of them. 

The first problem is on generalizing our theorem to the setting of C$^*$-algebras.
Kadison's theorem shows the following. 
\begin{proposition}
Let $A$ and $B$ be unital C$^*$-algebras. 
Suppose that $\phi\colon A_{sa}\to B_{sa}$ is a linear order isomorphism. 
Then there exist a Jordan $^*$-isomorphism $J\colon A\to B$ and $x\in B^{-1}$ such that $\phi(a) = xJ(a)x^*$ for every $a\in A_{sa}$.  
\end{proposition}
\begin{proof}
Take a positive number $k$ such that $k\geq\phi^{-1}(1)$. 
Then $\phi(k)\geq 1$, thus $\phi(1)\, (\geq 1/k)$ is an invertible operator in $B$. 
Define a mapping $J\colon A_{sa}\to B_{sa}$ by $J(a):= \phi(1)^{-1/2}\phi(a)\phi(1)^{-1/2}$, $a\in A_{sa}$. 
Then $J$ is a unital linear order isomorphism. 
It follows by Theorem \ref{Kadison} (due to Kadison) that $J$ extends to a Jordan $^*$-isomorphism from $A$ onto $B$. 
\end{proof}

Thus the following problem is of interest. 
\begin{problem}
Let $A$ and $B$ be two C$^*$-algebras and suppose that $\phi\colon A_{sa}\to B_{sa}$ is an order isomorphism. 
When does it follow automatically that $\phi$ is an affine mapping?
\end{problem}
In the setting of von Neumann algebras, the only exception was commutative (type I$_1$) cases (Theorem \ref{rest}(1)). 
Is it also true for C$^*$-algebras?

It seems to be highly challenging to give a general form of order isomorphisms between operator intervals in C$^*$-algebras. 
The following problem is motivated by Theorem \ref{Jordan}.
\begin{problem}
Let $A$ and $B$ be two unital C$^*$-algebras.  
Suppose that $\phi\colon \{a\in A_{sa}\mid 0\leq a\leq 1\}\to \{b\in B_{sa}\mid 0\leq b\leq 1\}$ is an order isomorphism with $\phi(1/2)=1/2$. 
When does $\phi$ admit an extension to some Jordan $^*$-isomorphism from $A$ onto $B$?
\end{problem}

In another direction, we can consider the following problem:
\begin{problem}
Let $M$ and $N$ be two von Neumann algebras. 
Suppose that $\phi\colon M_{*sa}\to N_{*sa}$ is an order isomorphism between self-adjoint parts of preduals. 
Can we characterize such a mapping by means of some Jordan $^*$-isomorphism from $M$ onto $N$? 
\end{problem}
Note that this problem was partly considered by Araki in \cite{A} with a much stronger assumption. 
His additional assumption is that $M$ is of type I and without I$_2$ direct summands, $\phi$ is continuous, and $\phi$ preserves linear structure and Jordan decompositions. 

More generally, one may also consider order isomorphisms of noncommutative $L^p$-spaces. 
For results on linear order isomorphisms between noncommutative $L^2$-spaces, see \cite[{Th\'eor\`{e}me 3.3}]{C} by Connes and \cite{Mi} by Miura. 
In the commutative case $M=L^{\infty}(\mu)$ for a measure $\mu$, $L^p(\mu)_{sa}$ ($1\leq p <\infty$) is order isomorphic to $L^1(\mu)_{sa}$ by the mapping $f\mapsto f\lvert f\rvert^{p-1}$. 
In finite dimensional cases, noncommutative $L^p$-spaces are isomorphic to the underlying von Neumann algebra as ordered vector spaces. 
Then, what is the general form of an order isomorphism from an interval of a  noncommutative $L^p$-space onto an interval of a noncommutative $L^q$-space in the setting of general von Neumann algebras? 

\medskip\medskip

\textbf{Acknowledgements} \quad 
The author is grateful to Yasuyuki Kawahigashi, the advisor of the author, for his invaluable support, and to Gilles Pisier and the anonymous referee for helpful comments. 
This work was supported by Leading Graduate Course for Frontiers of Mathematical Sciences and Physics, MEXT, Japan.

\end{document}